\documentclass[12pt]{article}
\usepackage{amsfonts,amssymb,amsthm,amsmath,latexsym,mathtools, bm}
\usepackage{cite}
\usepackage{hyperref}
\usepackage{array}
\usepackage{breqn}
\usepackage{color}
\usepackage{fancyhdr}
\usepackage{enumitem}
\usepackage{dsfont}
\usepackage{caption}
\usepackage{subcaption}
\usepackage{booktabs}
\usepackage{float}
\usepackage {extarrows}
\usepackage{verbatim}
\usepackage{arydshln}
\usepackage[margin=1.2in]{geometry}

\usepackage{youngtab}
\usepackage{ytableau}
\ytableausetup
{mathmode, boxsize=1.5em}

\usepackage{tikz}
\usetikzlibrary{cd}

\usepackage{times}

\allowdisplaybreaks
\def\qed{\nopagebreak\hfill{\rule{4pt}{7pt}}}
\def\proof{\noindent {\it{Proof.} \hskip 2pt}}
\newtheorem{theorem}{Theorem}[section]
\newtheorem{definition}[theorem]{Definition}
\newtheorem{lemma}[theorem]{Lemma}

\newtheorem{conjecture}[theorem]{Conjecture} 
\numberwithin{equation}{section}

\def\CAsc{{\rm CAsc}}
\def\casc{{\rm casc}}
\def\CDes{{\rm CDes}}
\def\cdes{{\rm cdes}}
\def\Des{{\rm Des}}
\def\des{{\rm des}}

\newcommand{\scrs}{{\mathcal{S}}}
\newcommand{\scrb}{{\mathcal{B}}}
\newcommand{\scrv}{{\mathcal{V}}}
\newcommand{\scre}{{\mathcal{E}}}
\newcommand{\stirlingsubset}[2]{\genfrac{\{}{\}}{0pt}{}{#1}{#2}}
\newcommand{\stirlingcycle}[2]{\genfrac{[}{]}{0pt}{}{#1}{#2}}

\newcommand{\euler}[2]{\genfrac{\langle}{\rangle}{0pt}{}{#1}{#2}}
\newcommand{\eulersecond}[2]{\left\langle\!\euler{#1}{#2} \!\right\rangle}

\begin{document}

\begin{center}
{\Large\bf The Higher-order Stirling Triangles}
 
William Y.C. Chen$^1$ and Elena L. Wang$^2$

$^1$Center for Applied Mathematics, KL-AAGDM\\
Tianjin University\\
Tianjin 300072, P.R. China

$^2$Chern Institute of Mathematics and LPMC\\
Nankai University\\
Tianjin 300071, P.R. China

\vskip 3mm

Emails: { $^1$chenyc@tju.edu.cn, $^{2}$ling\_wang2000@tju.edu.cn}

\end{center}

\begin{abstract}
    The $r$th-order Stirling cycle and subset triangles and their associated quasi-Eulerian triangles were introduced by Deb and Sokal in their study of total positivity of combinatorial triangles. They found combinatorial interpretations for the cycle case in terms of Stirling permutations, leaving the subset case open. For $r\ge 2$, we resolve this problem by introducing the notion of Stirling subset permutations along with a  consecutive-descent statistic. We also prove the conjectures of Deb and Sokal on the row log-concavity of the higher-order Stirling cycle and subset triangles. Our log-concavity proofs rely on Sagan's criterion, Dey's extension, and strengthened log-concavity inequalities discovered with the assistance of ChatGPT 5.6. 
\end{abstract}



\section{Introduction}
\label{sec:intro}

Many classical triangular arrays in enumerative combinatorics 
exhibit total positivity properties; see, for example, the
binomial~\cite{Aigner-Ziegler-2018}, the Stirling~\cite{Brenti-1995,Chen-Deb-2021,Harper-1967,Sokal-2023,Sokal-Zeng-2022}, and the Eulerian triangles~\cite{Brenti-1989,Brenti-1996,Chen-Deb-2021,Frobenius-1910,Petreolle-etal-2023, Sokal-Zeng-2022}. Motivated by these phenomena, Deb and Sokal~\cite{Deb-Sokal} introduced, for
 $r\ge1$, the $r$th-order Stirling cycle and subset triangles,
obtained by lifting
the columns of the $r$-associated Stirling triangles. Following the terminology in~\cite{Deb-Sokal,Graham-etal}, we refer to their entries as the $r$th-order Stirling cycle numbers and the $r$th-order Stirling subset numbers, respectively. We study combinatorial interpretations of the higher-order
Stirling triangles and their associated quasi-Eulerian triangles.
For the Stirling triangles, we also investigate row log-concavity
and the real-rootedness of the row-generating polynomials.

Recall that the $r$-associated Stirling cycle number
$\stirlingcycle{n}{k}_r$ counts permutations of
$[n]=\{1,2,\ldots,n\}$ with $k$ cycles, each of size at least $r$,
whereas the $r$-associated Stirling subset number
$\stirlingsubset{n}{k}_r$ counts partitions of $[n]$ into $k$ blocks,
each of size at least $r$. Introduced by Comtet~\cite{Comtet_74}, 
these numbers are also called the $r$-associated Stirling numbers of the
first and second kinds, respectively. 
In the literature, $\stirlingsubset{n}{k}_r$ is often denoted by
$S_r(n,k)$, and the signed version of $\stirlingcycle{n}{k}_r$ by
$s_r(n,k)$; see~\cite{Comtet_74}.

For $r\ge2$, each of the matrices
$(\stirlingcycle{n}{k}_r)_{n,\,k\ge0}$ and
$(\stirlingsubset{n}{k}_r)_{n,\,k\ge0}$
contains zero entries on its main diagonal. Deb and Sokal~\cite{Deb-Sokal} then considered the infinite
matrices with nonzero diagonal entries by lifting the columns. More precisely, for $r \ge 1$,
they defined the $r$th-order Stirling
cycle and subset numbers by
\begin{align}
\stirlingcycle{n}{k}^{\!(r)}
=
\stirlingcycle{n+(r-1)k}{k}_r,
\qquad
\stirlingsubset{n}{k}^{\!(r)}
=
\stirlingsubset{n+(r-1)k}{k}_r.
\end{align}
The resulting $r$th-order Stirling cycle and subset triangles are
\begin{align*}
C^{(r)}
=
\left(\stirlingcycle{n}{k}^{\!(r)}\right)_{n,\,k\ge0},
\qquad
S^{(r)}
=
\left(\stirlingsubset{n}{k}^{\!(r)}\right)_{n,\,k\ge0},
\end{align*}
with the zeroth row understood to be $(1,0,0,\ldots)$.

For $r=1$, $C^{(1)}$ and $S^{(1)}$ 
reduce to the ordinary Stirling cycle and subset triangles.
For $r=2$, the entries of $C^{(2)}$ are 
called the Ward cycle numbers; see OEIS sequences A259456 and A269940~\cite{OEIS}, whereas the entries of $S^{(2)}$
are known as the Ward numbers or Ward set numbers; see OEIS sequences A134991 and A269939~\cite{OEIS}. These Ward set numbers were studied by Jordan~\cite{Jordan-1933} and Ward~\cite{Ward-1934}
in connection with representations of Stirling numbers as sums
of factorials, and have since appeared in a variety of
contexts~\cite{Barbero-etal, Carlitz-1971, Clark_99, Steel, Price-Sokal-2020, Deb-Sokal, Wang-Xin}.  Write the row-generating polynomials of $C^{(r)}$ and $S^{(r)}$ as 
\begin{align*}
c_{r,\,n}(x)
=\sum_{k=0}^{n}\stirlingcycle{n}{k}^{\!(r)}x^k, \qquad
s_{r,\,n}(x)
=\sum_{k=0}^{n}\stirlingsubset{n}{k}^{\!(r)}x^k. 
\end{align*}

Building on the 
$r$th-order Stirling cycle and subset
numbers, 
Deb and Sokal~\cite{Deb-Sokal} introduced 
the $r$th-order quasi-Eulerian cycle
and subset triangles, 
\begin{align*}
     Q_{\rm C}^{(r)} = \left(\eulersecond{n}{k}^{\!(r)}_{\!\rm C}\right)_{n,\,k\ge 0},
\qquad Q_{\rm S}^{(r)}=\left(\eulersecond{n}{k}^{\!(r)}_{\!\rm S}\right)_{n,\,k\ge 0},
\end{align*} 
where the entries are defined via a binomial transformation: for $n \ge 0$ and $0 \le j \le n$, 
\begin{align}
\stirlingcycle{n}{n-j}^{(r)}
&=\sum_{i=j}^{n}
\eulersecond{n}{i}^{\!(r)}_{\!\rm C}\binom{i}{j},
\label{eq:quasi-Eulerian.cycle}\\[6pt]
\stirlingsubset{n}{n-j}^{(r)}
&=\sum_{i=j}^{n}
\eulersecond{n}{i}^{\!(r)}_{\!\rm S}\binom{i}{j}.
\label{eq:quasi-Eulerian.subset}
\end{align}

As pointed out by Deb and Sokal, the terminology of the $r$th-order quasi-Eulerian cycle and subset triangles is motivated by their connection to the second-order Eulerian triangles,
which can be understood as a relationship between
the Ward numbers and the second-order Eulerian numbers.

Recall that the second-order Eulerian numbers $\eulersecond{n}{k}$ 
are concerned with the enumeration of Stirling permutations of $[n]_2$, where $[n]_2=\{1,1,2,2,\ldots,n,n\}$. 
The notion of Stirling permutations was introduced by Gessel and
Stanley~\cite{Gessel-Stanley}. A Stirling permutation of $[n]_2$
is defined as a permutation of $[n]_2$ such that all elements between
two occurrences of $i$ are greater than $i$. Notice that a Stirling
permutation on a multiset is defined to be a permutation 
for which the elements between two occurrences of $i$ are at least $i$. 
For a Stirling permutation $\sigma=\sigma_1\sigma_2\cdots\sigma_{2n}$, an index $i\in[2n-1]$ is called an ascent if
$\sigma_i<\sigma_{i+1}$, and a descent if
$\sigma_i>\sigma_{i+1}$. With this convention, $\eulersecond{n}{k}$ is the number of
Stirling permutations of $[n]_2$ with $k$ ascents (or $k$ descents). Throughout this paper, ascents and descents are counted only at internal positions and we do not patch zeros at the beginning or at the end of a permutation. The $r$th-order Eulerian numbers count Stirling permutations on $[n]_r=\{1^r,2^r,\ldots,n^r\}$ by their number of ascents.

For $r=1$, both quasi-Eulerian triangles $Q_{\rm C}^{(1)}$ and $Q_{\rm S}^{(1)}$ contain negative entries. 
For $r=2$, however, the quasi-Eulerian cycle triangle coincides with the
second-order Eulerian triangle, while the quasi-Eulerian
subset triangle is a shifted reversal:
\begin{align}
\eulersecond{n}{k}^{\!(2)}_{\!\rm C}&=\eulersecond{n}{k},
\label{cycle.r2}\\[9pt]
\eulersecond{n}{k}^{\!(2)}_{\!\rm S}&=
\eulersecond{n}{n-k-1},
\label{subset.r2}
\end{align}
with the zeroth row understood to be $(1,0,0,\ldots)$.
 An equivalent form of \eqref{cycle.r2}  can be found in \cite[A106828/A358622]{OEIS}, and the relation \eqref{subset.r2} is also known; see  \cite{Smiley,Barbero-etal}. Notice that \eqref{subset.r2} 
 can be rewritten as follows:
\begin{align}
W(n,k)
=
\sum_{i=0}^{n-1}
\eulersecond{n}{i}\binom{n-i-1}{n-k},
\label{eq.Ward}
\end{align}
where $n\ge1$, $k\ge 0$,  and $W(n,k)$ is the Ward set number, 
namely, the number of partitions of $[n+k]$ with $k$ blocks such that
each block contains at least $2$ elements.
Deb and Sokal~\cite{Deb-Sokal} gave
algebraic proofs of both
\eqref{cycle.r2} and \eqref{subset.r2} along with a combinatorial proof of \eqref{cycle.r2}. 

For $r\ge2$, via computations, Deb and Sokal~\cite{Deb-Sokal} 
observed that for both the cycle and subset cases,
the $r$th-order quasi-Eulerian entries are nonnegative. They provided combinatorial interpretations
for the cycle case for $r \ge 2$, which ensures their nonnegativity.
In fact, their combinatorial structure stems from the objects defined by them to understand the $r$th-order Stirling cycle triangle. More precisely, both the Stirling cycle numbers and the associated
quasi-Eulerian cycle numbers are interpreted in terms of Stirling permutations of a multiset
$M_n^{(r)}$ and a consecutive-ascent statistic.
For $r\ge2$ and $n\ge1$, the multiset is defined by
\begin{align*}
M_n^{(r)}
=
\{1^{a_1},2^{a_2},\ldots,((r-1)n)^{a_{(r-1)n}}\},
\end{align*}
where the multiplicity $a_i$ of $i$ is given by
\begin{align*}
a_i=
\begin{cases}
2, & \text{if } i\equiv0\pmod{r-1},\\
1, & \text{otherwise}.
\end{cases}
\end{align*}

As remarked by Deb and Sokal~\cite[Section~4.4]{Deb-Sokal}, 
it is natural to wonder whether
there is a combinatorial interpretation of 
the $r$th-order Stirling subset triangle analogous to the cycle case. They also left open the problem of finding a combinatorial interpretation of the $r$th-order quasi-Eulerian subset triangle~\cite[Section~1.4]{Deb-Sokal}.

For $r \ge 2$, we answer both questions by introducing a class of Stirling permutations on $M_n^{(r)}$, called
Stirling subset permutations, together with a consecutive-descent statistic. 
The $(n,k)$-entry of the $r$th-order Stirling subset triangle
is a weighted count of these Stirling subset permutations of $M_n^{(r)}$: a permutation
with $i$ consecutive descents is endowed with the weight
$\binom{i}{n-k}$, the number of ways to choose $n-k$ of these consecutive descents.
The $(n,k)$-entry of the associated quasi-Eulerian subset triangle equals the number of Stirling subset permutations of $M_n^{(r)}$ with $k$ consecutive descents.

When $r=2$, our interpretation of the second-order Stirling subset numbers reduces to the combinatorial interpretation of the Ward numbers
$W(n,k)$ in terms of a weighted 
counting of Stirling permutations of $[n]_2$; see also Barbero, Salas and
Villase\~nor~\cite{Barbero-etal}. Moreover, our interpretation of the second-order quasi-Eulerian subset triangle gives a combinatorial interpretation of~\eqref{subset.r2}.

The second part of this paper is devoted to several conjectures of Deb and Sokal~\cite{Deb-Sokal}
on the real-rootedness of the row-generating polynomials and
the row log-concavity of $r$th-order Stirling cycle and subset
triangles. 

Here and throughout, a polynomial with real coefficients is said to be negative-real-rooted if it is either identically zero or all its  zeros lie in $(-\infty, 0]$. For $r=1$, negative-real-rootedness is immediate in the
cycle case, while the subset case was proven by Harper~\cite{Harper-1967}. For $r=2$, the second-order Stirling subset polynomials are the Ward polynomials, which have only real zeros; see Clark~\cite[Lemma~1]{Clark_99}.
 The real-rootedness of the second-order cycle polynomials
 was established by Deb and Sokal~\cite[Theorem~1.6 and Section~6]{Deb-Sokal}.

For higher orders, Deb and Sokal~\cite[Section~7.1]{Deb-Sokal}  proved
the non-real-rootedness of $c_{r,\,n}(x)$ whenever $r\ge3$ and
$n\ge3$, and of $s_{r,\,n}(x)$ whenever $r\ge4$ and $n\ge3$.
They~\cite[Conjecture~1.4(c) and Section~7.1]{Deb-Sokal} conjectured that $s_{3,\,n}(x)$ has non-real zeros for $n \ge 4$. 
We confirm this assertion in Theorem~\ref{theorem:S3-nonreal}.

We also settle the row log-concavity conjectures of Deb and Sokal and obtain a further result for the sixth-order Stirling subset triangle. 
Although real-rootedness fails to hold for higher orders in general, Deb and Sokal~\cite[Conjecture~1.3(c) and Conjecture~1.4(c)]{Deb-Sokal} conjectured that the rows of both triangles remain log-concave
for $3\le r\le5$. The conjecture for the
cycle case was posed by Deb at a conference; see~\cite{Cameron}. We prove these conjectures for both families: every row of
$C^{(r)}$ and every row of 
$S^{(r)}$ is log-concave for $1\le r\le5$. Furthermore, we show that for $n\ge23$, the $n$th row of $S^{(6)}$ is log-concave.

The rest of the paper is organized as follows.
In Section~\ref{sec.interpretation}, we give an overview of
the interpretations of Deb and Sokal for the cycle case and present the corresponding interpretations for the subset case.
In Section~\ref{sec:row-log-concavity}, we prove row log-concavity for both Stirling families up to order five. We also show that every row with $n \ge 23$ in the sixth-order subset triangle is log-concave. In Section~\ref{sec:S3-nonreal}, we focus on the non-real-rootedness of the third-order
Stirling subset polynomials.
We conclude with a conjecture on row log-concavity and log-convexity in higher orders.

\section{Combinatorial interpretations}
\label{sec.interpretation}

In this section, we first recall the combinatorial interpretations of Deb and Sokal~\cite{Deb-Sokal} for the
$r$th-order Stirling cycle triangle and the
$r$th-order quasi-Eulerian cycle triangle for $r \ge 2$. 
We then answer the questions of Deb and Sokal concerning combinatorial interpretations for the subset case by introducing Stirling subset
permutations and consecutive descents.

The recurrence relations for the
$r$-associated Stirling numbers can be reformulated into the
following
recurrences~\cite{Comtet_74,Deb-Sokal} for the $r$th-order
Stirling cycle and subset numbers with $r \ge 1$. For $n \ge 1$ and $0 \le k \le n$,
\begin{align}
&\begin{aligned}
   \stirlingcycle{n}{k}^{\!(r)} &=
   (r-1)! \, \binom{n+(r-1)k-1}{r-1} \, \stirlingcycle{n-1}{k-1}^{(r)} \\[6pt]
        & \qquad \qquad +
       \big(n+(r-1)k-1\big)\, \stirlingcycle{n-1}{k}^{(r)},
   \label{rec.cycle}
   \end{aligned}
   \\[6pt]
   &\begin{aligned}
   \stirlingsubset{n}{k}^{\!(r)}
   &=
   \binom{n+(r-1)k-1}{r-1} \, \stirlingsubset{n-1}{k-1}^{(r)} 
         +  k\, \stirlingsubset{n-1}{k}^{(r)},
   \label{rec.subset}
   \end{aligned}
\end{align}
where 
\begin{align*}
    \stirlingcycle{0}{k}^{(r)}=\stirlingsubset{0}{k}^{(r)} = \delta_{k,\,0}, \qquad \stirlingcycle{n}{-1}^{(r)}=\stirlingsubset{n}{-1}^{(r)} = 0.
\end{align*}   

\subsection{The Deb--Sokal construction for the cycle case}
\label{sec.interpretation.cycle}

 For $r \ge 2$,  Deb and Sokal \cite{Deb-Sokal} 
   provided an interpretation for the 
$r$th-order Stirling cycle numbers $\stirlingcycle{n}{k}^{\!(r)}$
in terms of a weighted counting
of Stirling permutations with respect to the
number of consecutive ascents.   They also
provided an interpretation for the $r$th-order quasi-Eulerian cycle numbers $\eulersecond{n}{k}^{\! (r)}_{\!\rm C}$ in the same setting.
 The proof of the case for the second-order Stirling cycle numbers is due to Callan; see \cite[pp. 16--17]{Deb-Sokal}.  

Recall that the multiset introduced by Deb and Sokal is given by
\begin{align*}
M_{n}^{(r)} = \{1^{a_1}, 2^{a_2}, \ldots, ((r-1)n)^{a_{(r-1)n}}\},
\end{align*}
where $a_i$ is the multiplicity of the element $i$, defined by
\begin{align*}
a_i=
\begin{cases}
2, \qquad \text{if $i\equiv 0 \pmod{r-1}$;}\\
1, \qquad \text{otherwise.}
\end{cases}
\end{align*}
In particular, for $r=2$,  
\[ M_n^{(2)} =  [n]_2=\{1,1,2,2,3,3,\ldots, n,n\}\] 
and for $r=3$, 
\[ M_n^{(3)}=\{1,2,2,3,4,4,\ldots,2n-1 ,2n,2n\}.\]

Let $\scrs_{n}^{(r)}$ denote the set of Stirling permutations of the multiset $M_{n}^{(r)}$. For $n=0$, assume that $M_0^{(r)}=\varnothing$ and $\scrs_0^{(r)}$ consists of the empty word, with no consecutive ascents. Deb and Sokal~\cite{Deb-Sokal} introduced a new combinatorial statistic on $\scrs_{n}^{(r)}$.        

\begin{definition}[Consecutive ascent]
Given $\sigma \in \scrs_n^{(r)}$, a segment $\sigma_j \sigma_{j+1} \cdots \sigma_{j+r-1}$ is called a consecutive ascent if $\sigma_{j}<\sigma_{j+1}$
and  $\sigma_{j+1} \cdots \sigma_{j+r-1}$ is an increasing sequence
of consecutive integers from $ (r-1)i+1$ to $(r-1)i+(r-1)$ 
        for some $ 1 \leq i \le n-1$. 
\end{definition}        

It can be seen that for each consecutive ascent, the integer $i$ in the definition is uniquely determined and distinct consecutive ascents have distinct indices. Thus we can represent the set of consecutive ascents by the indices, as used in \cite{Deb-Sokal}.  We call the integer $i$ the index of the consecutive ascent. 
Let $\CAsc(\sigma) \subseteq \{1,2,\ldots,n-1\}$ denote the set of indices of consecutive ascents of $\sigma$, and let  $\casc(\sigma)$ denote its cardinality.

For example, when $r=3$ and $n=3$, the Stirling permutation $\sigma =221344566$ of $M_3^{(3)} =\{1,2,2,3,4,4,5,6,6\}$ has $2$ consecutive ascents, $134$ and $456$, corresponding to the indices $1$ and $2$.

Note that it is readily seen from the property of Stirling permutations that the above description
of a consecutive ascent is equivalent to the definition in Deb and Sokal~\cite{Deb-Sokal}, whether it is specified by its index or by the corresponding segment.  

Put
\begin{align}
\scrs_{n,\,k}^{(r)}
=
\{\sigma \in \scrs_{n}^{(r)} \mid  \casc(\sigma)= k\}.
\end{align}
Let $\scrv_{n,\,k}^{(r)}$ denote the set of pairs
$(\sigma,D)$ such that $\sigma\in\scrs_n^{(r)}$,
$D\subseteq\CAsc(\sigma)$ and $|D|=k$. 
In the terminology of Deb and Sokal \cite{Deb-Sokal},
these pairs are called consecutive-ascent-marked Stirling permutations. 

\begin{theorem}[Deb--Sokal]
\label{thm.combi.interpret.rth}
Given $r \ge 2$, the $r$th-order quasi-Eulerian cycle number $\eulersecond{n}{k}^{\!(r)}_{\!\rm C}$  and  the $r$th-order Stirling cycle number $\stirlingcycle{n}{k}^{\!(r)}$ have the following interpretations. For $n \ge 0$ and $k \ge 0$, we have
\begin{align}
\eulersecond{n}{k}^{\!(r)}_{\!\rm C}=|\left\{\sigma\in\scrs_n^{(r)} \mid  \casc(\sigma)=k\right\}|,
\label{eq:intro-QC}
\end{align}
and
\begin{align}
\stirlingcycle{n}{k}^{\!(r)}
=
|\left\{
(\sigma,D) \mid
\sigma\in\scrs_n^{(r)},
D\subseteq\CAsc(\sigma),
|D|=n-k\right\}|.
\label{eq:intro-C-marked}
\end{align}
\end{theorem}

Notice that for $r=2$, $\scrs_{n}^{(r)}$ is the set of Stirling permutations of $\{1,1,2,2,\ldots,n,n\}$, 
and the number of consecutive ascents coincides with the number of standard ascents in Stirling permutations. 
Hence, Theorem~\ref{thm.combi.interpret.rth} implies \eqref{cycle.r2}, namely,
\begin{align*}
    \eulersecond{n}{k}^{\!(2)}_{\!\rm C} = \eulersecond{n}{k},
\end{align*}
where $\eulersecond{n}{k}$ is the second-order Eulerian number.

\subsection{The subset case}
\label{sec.interpretation.subset}

As to combinatorial interpretations of
the $r$th-order Stirling subset numbers and the
$r$th-order quasi-Eulerian subset numbers for $r \ge 2$, we will restrict our attention to a class of Stirling permutations of $M_n^{(r)}$,
called Stirling subset permutations, which are defined as follows.

First, write
\begin{align} 
M_n^{(r)} = M_{n,\,1}^{(r)} \cup   M_{n,\,2}^{(r)} \cup \cdots  \cup
M_{n,\,n}^{(r)}, 
\end{align}
where for each $1 \le i \le n$,
\begin{align}
M_{n,\,i}^{(r)} = \{(r-1)(i-1)+1, (r-1)(i-1)+2, \ldots, (r-1)i, (r-1)i\}.
\end{align}

A Stirling permutation $\sigma$ of $M_n^{(r)}$ is said to be
a Stirling subset permutation of $M_n^{(r)}$ if for each $i$,
the elements in $M_{n,\,i}^{(r)}$ appear in non-increasing order
(from left to right) in $\sigma$. For example, 
$244321$ is a Stirling subset permutation on $M_2^{(3)}$, where $M_{2,\,1}^{(3)}=\{1,2,2\}$ and $M_{2,\,2}^{(3)}=\{3,4,4\}$.

Let $\scrb_{n}^{(r)}$ denote the set of  
 Stirling subset permutations of $M_n^{(r)}$. Note that  $\scrb_{n}^{(2)} = \scrs_{n}^{(2)}$ and $\scrb_{n}^{(r)} \subset \scrs_{n}^{(r)}$ for $r \ge 3$ and $n \ge 1$.
By analogy with the definition of consecutive ascents described in Section~\ref{sec.interpretation.cycle}, we define the consecutive descents of $\sigma\in \scrb_n^{(r)}$.

\begin{definition}
Given $\sigma\in \scrb_n^{(r)}$, a segment $\sigma_{j} \cdots \sigma_{j+r-2} \sigma_{j+r-1}$ is called a consecutive descent if the 
following two conditions hold. 
\begin{enumerate}
    \item 
The segment $\sigma_{j} \cdots \sigma_{j+r-2}$ is the decreasing
sequence
of consecutive integers from $ (r-1)i$ to $(r-1)i-(r-2)$ for some $2 \leq i \le n$ and $\sigma_{j+r-2}>\sigma_{j+r-1}$.

\item 
Assume that
$\sigma_{j+r-1}$  belongs to a block $M^{(r)}_{n,\,k}$.
Then the rest of the elements in
$M^{(r)}_{n,\,k}$ appear after $\sigma_{j+r-1}$ in  $\sigma$. 
\end{enumerate}
\end{definition}  

Notice that in the above
definition, by the defining property of Stirling permutations, the first element $\sigma_j$ of a consecutive descent $\sigma_j\cdots\sigma_{j+r-1}$ is necessarily its second occurrence in $\sigma$. Moreover, since the elements of each block $M_{n,\,i}^{(r)}$ 
appear in non-increasing order in $\sigma$, $\sigma_{j+r-1}$ must be the maximum of its underlying block, i.e.,  the
element $(r-1)k$ of $M_{n,\,k}^{(r)}$.

As in the case of consecutive ascents, for a consecutive descent $\sigma_{j} \cdots \sigma_{j+r-2} \sigma_{j+r-1}$ with $\sigma_j=(r-1)i$,  we use $i$ to represent this consecutive descent.
Clearly, the integer $i$  is uniquely determined and different consecutive descents must have distinct indices. Thus the indices of the consecutive descents form a subset of  $\{2,3,\ldots,n\}$. 

Write $\CDes(\sigma)$ for the set of indices of consecutive descents
of $\sigma$ and $\cdes(\sigma)$ for its cardinality.
Define $\scrb_{n,\,k}^{(r)}$ to be the set of
Stirling subset permutations of $M_n^{(r)}$ with $k$ consecutive descents:
\begin{align}
\scrb_{n,\,k}^{(r)}
=
\{\sigma \in \scrb_{n}^{(r)} \mid \cdes(\sigma) = k\}.
\end{align}
We adopt the convention that $\scrb_0^{(r)}$ consists of the empty word with 
no consecutive descents.
For example, when $r=3$ and $n=2$, we have $M_2^{(3)}=\{1,2,2\}\cup\{3,4,4\}$, 
\begin{align*}
\scrb_{2,\,0}^{(3)}
=
\{&
221443,\ 224413,\ 224431,\ 244213,\ 244231,\\
&244321,\ 442213,\ 442231,\ 442321
\},
\end{align*}
and $\scrb_{2,\,1}^{(3)}
=\{443221\}$.

Let $\scre_{n,\,k}^{(r)}$ denote the set of pairs $(\sigma,D)$ with $\sigma\in\scrb_n^{(r)}$, $D\subseteq\CDes(\sigma)$ and $|D|=k$:
\begin{align} \label{pairs}
\scre_{n,\,k}^{(r)}=\left\{
(\sigma,D) \mid
\sigma\in\scrb_n^{(r)},
D\subseteq\CDes(\sigma),
|D|=k\right\}.
\end{align} 
Denote $e_{n,\,k}^{(r)} = | \scre_{n,\,k}^{(r)} |$
and $b_{n,\,k}^{(r)} = | \scrb_{n,\,k}^{(r)} |$. By definition,
\begin{align}
e_{n,\,k}^{(r)}
=
\sum_{i=k}^{n} b_{n,\,i}^{(r)} \binom{i}{k}.
\label{eq.enkr.eulersecond}
\end{align}

The above description of $\scre_{n,\,k}^{(r)}$ can be 
rephrased in terms of enriched Stirling subset permutations on $M_n^{(r)}$, which
we call consecutive-descent-marked Stirling subset permutations,
in alignment with the term consecutive-ascent-marked Stirling
permutations introduced by Deb and Sokal~\cite{Deb-Sokal}. 
Specifically, for each consecutive descent, we insert a down
arrow immediately before its last entry if its index belongs
to $D$, and a bar otherwise. In this notation, 
$\scre_{n,\,k}^{(r)}$ can be viewed as the set of consecutive-descent-marked Stirling subset permutations on $M_n^{(r)}$ with $k$ down arrows. For example, an element $(887664543221,\{4\}) \in \scre_{4,\,1}^{(3)}$ can be represented as follows:
\begin{align}
8\; 8\;7\downarrow6\;6\;4\;5\;4\;3\mid2\;2\;1. 
\end{align}

We are now ready to present the combinatorial interpretations of 
the $r$th-order Stirling subset numbers and 
the $r$th-order quasi-Eulerian subset numbers. 

\begin{theorem}
\label{thm:subset-interpretation}
Let $r\ge2$.  For $n\ge0$ and $k \ge 0
$, we have
\begin{align}
\eulersecond{n}{k}^{\!(r)}_{\!\rm S}=|\left\{\sigma\in\scrb_n^{(r)} \mid  \cdes(\sigma)=k\right\}|,
\label{eq:intro-QS}
\end{align}
and
\begin{align}
\stirlingsubset{n}{k}^{\!(r)}
=
|\left\{
(\sigma,D) \mid
\sigma\in\scrb_n^{(r)},
D\subseteq\CDes(\sigma),
|D|=n-k\right\}|.
\label{eq:intro-S-marked}
\end{align}
\end{theorem}

\proof
To prove the theorem, we first show that for $n \ge 1$ and $0\le k\le n$, the numbers $e_{n,\,k}^{(r)}$ satisfy the recurrence relation
\begin{align}
e_{n,\,k}^{(r)} 
	=
	\binom{r(n-1) \,-\, (r-1)(k-1)}{r-1}\, e_{n-1,\,k}^{(r)}
	+
	(n-k)\, e_{n-1,\,k-1}^{(r)} 
     \label{eq.rec.enkr}
\end{align}
with initial conditions $e_{0,\,k}^{(r)}=\delta_{k,\,0}$,
where $e_{n,\,k}^{(r)}=0$ for $k<0$ or $k>n$. After replacing $k$ by $n-k$, recurrence~\eqref{eq.rec.enkr} becomes precisely the recurrence for
$\stirlingsubset{n}{k}^{\!(r)}$, with the same initial conditions. Hence~\eqref{eq:intro-S-marked} follows. Combining~\eqref{eq:intro-S-marked} with the defining relation~\eqref{eq:quasi-Eulerian.subset} and~\eqref{eq.enkr.eulersecond} yields~\eqref{eq:intro-QS}.

Now we proceed to prove~\eqref{eq.rec.enkr}. Let $(\sigma,D) \in \scre_{n,\,k}^{(r)}$ be a
consecutive-descent-marked Stirling subset permutation.  Set $u_i = (r-1)n-i+2$ for  $2\leq i \leq r$
and $u_{1} = (r-1)n$.

Let $\sigma'$ be the permutation obtained from $\sigma$ by deleting all 
the elements of $M_{n,\,n}^{(r)}$, and put $D'=D\setminus\{n\}$.
Notice that every consecutive descent of $\sigma$ whose index $i$ belongs to $D'$ remains a consecutive descent of $\sigma'$ with the same index $i$.
Thus $(\sigma',D')$ belongs to $\scre_{n-1,\,k}^{(r)}$ if $n\notin D$, and to $\scre_{n-1,\,k-1}^{(r)}$ if $n\in D$. Note that if the deletion operation
creates additional consecutive descents, we just represent 
these newly formed consecutive descents by bars so that the consecutive descents marked by down arrows remain unchanged.  For example, for $r=3,n=4$, after deleting elements $8,8,7$ in $665\downarrow 443882217$, we obtain $665\downarrow 443 \mid 221$.

Conversely, given a consecutive-descent-marked Stirling subset permutation $(\sigma',D')$, we produce an element $(\sigma,D) \in \scre_{n,\,k}^{(r)}$ by inserting the $r$ elements 
$u_1, u_2,\ldots, u_{r}$. It is clear
the elements $u_1$ and $u_2$ must be adjacent in view of
the property of Stirling permutations.  Now we encounter two cases 
depending on whether $n$ is in $D$.

\noindent
Case 1: $n \notin D$. The underlying permutation $\sigma$ of $(\sigma,D) \in \scre_{n,\,k}^{(r)}$ can be viewed as a shuffle of the sequence $(u_1,\cdots,u_r)$ and the underlying permutation $\sigma'$ of $(\sigma',D') \in \scre_{n-1,\,k}^{(r)}$, where $D=D'$. Since $u_1$ and $u_2$ 
need to be adjacent, we may treat $u_1u_2$ as a single element, thereby regarding the sequence $(u_1,\cdots,u_r)$ as a sequence of $r-1$ elements. 

Before counting the admissible insertions, we show that distinct
consecutive descents of each Stirling subset permutation are disjoint. Suppose that $\sigma_j\cdots\sigma_{j+r-2}\sigma_{j+r-1}$ and $\sigma_t \cdots \sigma_{t+r-2}\sigma_{t+r-1}$ are two different consecutive descents of $\sigma$. Without loss of generality, we may assume that $j<t$. 

We first notice that the elements $\sigma_{j+1}, \ldots, \sigma_{j+r-2}$, if any, all appear in
  a block $M_{n,\,i}^{(r)}$ and the elements $\sigma_{t+1}, \ldots, \sigma_{t+r-2}$, if any, all appear in
  another block $M_{n,\,i'}^{(r)}$ with $i' \neq i$. 
Moreover, we have  $j+r-1\not= t$, that is, the last element of the consecutive descent
$\sigma_j \cdots \sigma_{j+r-2}\sigma_{j+r-1}$ cannot be the first element of another consecutive descent.
 Since the first element $\sigma_t$ is the maximum element 
 of its corresponding block $M_{n,\,k}^{(r)}$ followed by a smaller element, it makes its second appearance in $\sigma$ by the Stirling property.  The last element $\sigma_{j+r-1}$ is also the maximum of its corresponding block $M_{n,\,k'}^{(r)}$, where $k'$ and
 $k$ are not necessarily distinct. But $\sigma_{j+r-1}$ is the first occurrence of its value in $\sigma$. So we conclude that the two
 consecutive descents do not overlap in $\sigma$. 

For example, when $r=3$ and $n=4$, take $(\sigma',D')=(665443221,\{2,3\})$. The following representation shows that the two consecutive descents corresponding to $i=3$ and $i=2$ are disjoint segments of $\sigma'$, although they contain the common value $4$:
\begin{align}
\label{ex.forbidden.insertion}
6\;\,\underbrace{6\cdot 5\cdot 4}_{i=3}\;\,\underbrace{4\cdot3\cdot2}_{i=2}\;\,2\;\,1.
\end{align}

In order to maintain the consecutive descents marked by down arrows after the insertion, the new element cannot be inserted into the positions between any two adjacent elements $\sigma_s$ and $\sigma_{s+1}$ in a consecutive descent marked by a down arrow. Therefore for such a consecutive descent, there are $r-1$ positions at which the insertion is not allowed. For example, the dots in~\eqref{ex.forbidden.insertion} represent the forbidden insertion positions in $\sigma'$. Note that we adopt the dot notation used by Deb and Sokal~\cite{Deb-Sokal} to represent consecutive-ascent-marked Stirling permutations.

For $(\sigma',D')$, since the consecutive descents are disjoint, contracting each consecutive descent marked by down arrows to a single element yields a sequence of length $\ell=r(n-1)-(r-1)k$. We then see that the total number of feasible insertions equals $ \binom{\ell+r-1}{r-1}$, which is the number of shuffles of two disjoint permutations of lengths $\ell$ and $r-1$, respectively.

\noindent
Case 2: $n \in D$. Given an element $(\sigma',D') \in\scre_{n-1,\,k-1}^{(r)}$, we insert the
largest block into $\sigma'$ and set $D=D'\cup\{n\}$. Since $u_1u_2$ must be adjacent,
the condition $n\in D$ forces all the $r$ inserted elements to
form a consecutive sequence $u_1u_2 \cdots u_r$. Moreover, if
$n$ is an index of a consecutive descent, then this consecutive
sequence must be inserted immediately before the first occurrence of the maximum element of an existing block $M_{n,\,i}^{(r)}$, where $1\le i\le n-1$. 
Each of the $k-1$ consecutive descents corresponding to the indices in $D'$ excludes exactly one position for the insertion: the position immediately before the final element of each consecutive descent. These excluded positions are distinct. Hence there are $(n-1)-(k-1)=n-k$ choices.

It can be seen that the deletion and the insertion procedures are
inverses of each other. This yields a bijection, and so we 
have completed the proof of the recurrence~\eqref{eq.rec.enkr}.

We now turn to the proof of the interpretation of the $r$th-order
Stirling subset numbers.
Set
\begin{align*}
f_{n,\,k}^{(r)}=e_{n,\,n-k}^{(r)}.
\end{align*}
Replacing $k$ by $n-k$ in \eqref{eq.rec.enkr} gives
\begin{align*}
f_{n,\,k}^{(r)}
=
\binom{n+(r-1)k-1}{r-1}f_{n-1,\,k-1}^{(r)}+k f_{n-1,\,k}^{(r)}.
\end{align*}
This is precisely the recurrence \eqref{rec.subset}
for $\stirlingsubset{n}{k}^{\!(r)}$.
Taking the initial values into consideration, we are led to  
\begin{align*}
\stirlingsubset{n}{k}^{\!(r)}
=
e_{n,\,n-k}^{(r)}.
\end{align*}
This proves~\eqref{eq:intro-S-marked}.

It follows from~\eqref{eq:intro-S-marked} that
\begin{align}
\sum_{k=0}^n
\stirlingsubset{n}{k}^{\!(r)}x^{n-k}
&=
\sum_{\sigma\in\scrb_n^{(r)}}
\ \sum_{D\subseteq\CDes(\sigma)}x^{|D|}
\notag\\[6pt]
&=
\sum_{\sigma\in\scrb_n^{(r)}}
(1+x)^{|\CDes(\sigma)|}.
\label{eq:subset-marking-transform}
\end{align}
On the other hand, multiplying 
the defining relation~\eqref{eq:quasi-Eulerian.subset}
by $x^j$ and summing over $0\le j\le n$, we get
\begin{align*}
\sum_{k=0}^n
\stirlingsubset{n}{k}^{\!(r)}x^{n-k}
&=
\sum_{k=0}^n
\eulersecond{n}{k}^{\!(r)}_{\!\rm S}
\sum_{i=0}^k\binom{k}{i}x^i\\[6pt]
&=
\sum_{k=0}^n
\eulersecond{n}{k}^{\!(r)}_{\!\rm S}(1+x)^k.
\end{align*}
Combining this identity with
\eqref{eq:subset-marking-transform} and replacing
$x$ by $x-1$, we obtain
\begin{align*}
\sum_{k=0}^n
\eulersecond{n}{k}^{\!(r)}_{\!\rm S}x^k
=
\sum_{\sigma\in\scrb_n^{(r)}}
x^{|\CDes(\sigma)|}.
\end{align*}
Equating coefficients of $x^k$ proves~\eqref{eq:intro-QS}.
\qed

We end this section with some special cases 
of the $r$th-order Stirling subset numbers. 
Notice that the Stirling subset permutations in $\scrb_n^{(r)}$ are
counted by the diagonal entries of $S^{(r)}$, i.e.,
\begin{align*}
|\scrb_{n}^{(r)}|=\stirlingsubset{n}{n}^{(r)}  = \frac{(rn)!}{(r!)^n n!}.
\end{align*}

For $r=2$,  the underlying multiset becomes
$[n]_2$, and $\scrb_n^{(2)}$ is precisely the set of ordinary Stirling permutations of $[n]_2$.
In this case, a  consecutive descent reduces to a $1$-descent  in the terminology
of Janson, Kuba and Panholzer~\cite{Janson-etal}. For $\sigma=\sigma_1\sigma_2\cdots\sigma_{2n}\in\scrb_n^{(2)}$,
let $\Des_1(\sigma)$ denote the set of positions $1\le i\le2n-1$
such that $\sigma_i>\sigma_{i+1}$ and $\sigma_{i+1}$ makes its
first appearance in $\sigma$. Write
$\des_1(\sigma)=|\Des_1(\sigma)|$.
The map $i\to\sigma_i$ is a bijection from
$\Des_1(\sigma)$ to $\CDes(\sigma)$. 

Since $\stirlingsubset{n}{k}^{\!(2)}$ equals the Ward number $W(n,k)$,
specializing~\eqref{eq:intro-S-marked} to $r=2$ and applying
the bijection from $\Des_1(\sigma)$ to $\CDes(\sigma)$ shows that $W(n,k)$ equals the number of ordinary Stirling
permutations of $[n]_2$ with $n-k$ $1$-descents marked by down arrows (in the sense of marked consecutive descents, as aforementioned); see also Barbero, Salas and Villase\~nor~\cite{Barbero-etal}.

With the interpretation as given in~\eqref{eq:intro-QS}, we can thus obtain a combinatorial proof for the identity 
\begin{align*}
\eulersecond{n}{k}^{\!(2)}_{\!\rm S}
=
\eulersecond{n}{n-k-1},
\end{align*}
as given in~\eqref{subset.r2}.

\section{The row log-concavity conjecture}
\label{sec:row-log-concavity}

In this section, we focus on the row log-concavity
of the $r$th-order Stirling triangles. For $r=1,2$, the log-concavity is known;
see~\cite{Deb-Sokal}.
For $3\leq r\le 5$, 
Deb and Sokal~\cite[Conjecture~1.3(c) and Conjecture~1.4(c)]{Deb-Sokal} conjectured that the row log-concavity holds 
for both the $r$th-order Stirling cycle triangle and the $r$th-order Stirling subset triangle.
The main objective of this section is to prove this conjecture. 
We also show that row log-concavity still holds for the sixth-order Stirling subset triangle for $n\geq 23$. 

Recall that for a nonnegative sequence $a=(a_0,a_1,\ldots)$, we say $a$ is log-concave if $a_k^2- a_{k-1}a_{k+1} \ge 0$ for all $k \ge 1$.  

\begin{theorem}
\label{thm:main-row-LC}
The following statements hold.
\begin{enumerate}
\item
For $3\le r\le5$, every row of both the $r$th-order Stirling cycle
triangle $C^{(r)}$ and the $r$th-order Stirling subset triangle $S^{(r)}$
is log-concave.

\item
For the sixth-order Stirling subset triangle $S^{(6)}$, every row with
$n\ge 23$ is log-concave.
\end{enumerate}
\end{theorem}

To prove the theorem, we need the following log-concavity criteria due to Sagan~\cite{Sagan1988} and
Dey~\cite{Dey}. For a sequence $s=(s_k)_{k\ge0}$, define for $k \ge 1$,
\begin{equation}
 \Delta_k(s)=s_k^2-s_{k-1}s_{k+1}.
\end{equation}
For two sequences $c=(c_k)_{k\ge0}$ and $d=(d_k)_{k\ge0}$, define for $k \ge 1$,
\begin{equation}
 \Gamma_k(c,d)
 =2c_kd_k-c_{k-1}d_{k+1}-d_{k-1}c_{k+1}.
\end{equation}

Sagan \cite{Sagan1988} found the following criterion  for  log-concavity.

\begin{lemma}[Sagan's criterion]
\label{lem:Sagan-log-concavity}
Let $(a_k)_{k\ge0}$ be a nonnegative log-concave sequence with no
internal zeros, with the convention that $a_k = 0$ for $k < 0$, and for $k \ge 0$, define 
\[
 b_k=c_ka_k+d_ka_{k-1},
\]
where $c_k,d_k\ge0$. Suppose that, for every $k \ge  1$,
\[
 \Delta_k(c) \ge 0,\qquad \Delta_k(d) \ge  0,
\]
and
\begin{equation}
 \Gamma_k(c,d) \ge 0.
 \label{eq:Sagan-cross}
\end{equation}
Then $(b_k)_{k\ge0}$ is log-concave.
\end{lemma}

Dey~\cite[Theorem~30]{Dey} 
extended Sagan's criterion by allowing the
mixed difference in~\eqref{eq:Sagan-cross} to be negative.  

\begin{lemma}[Dey's extension]
\label{lem:quadratic-log-concavity}
Let $(a_k)_{k\ge0}$ be a nonnegative log-concave sequence with no
internal zeros, with the convention that $a_k = 0$ for $k < 0$, and for $k\ge0$, define
\[
 b_k=c_ka_k+d_ka_{k-1},
\]
where $c_k,d_k\ge0$. Suppose that, for every $k\ge1$,
\[
 \Delta_k(c)\ge0,\qquad \Delta_k(d)\ge0,
\]
and, whenever $\Gamma_k(c,d)<0$,
\begin{equation}
 \Gamma_k(c,d)^2\le4\Delta_k(c)\Delta_k(d).
 \label{eq:quadratic-discriminant}
\end{equation}
Then $(b_k)_{k\ge0}$ is log-concave.
\end{lemma}

Given $r\ge 3$, let $T$ denote the triangle $C^{(r)}$ or $S^{(r)}$. We have $T_{0,\,k}=\delta_{k,\,0}$; for $n\ge1$, $T_{n,\,k} \neq 0$ if and only if $1\le k\le n$. Thus every row has no
internal zeros. For the $r$th-order Stirling triangles of both kinds, the rows $n=0,1,2$ are trivially log-concave, since each contains
at most two positive entries. Moreover, after each row is patched with zeros whenever the index is out of range, the correctness of the log-concavity inequalities at $k=1$ and $k\ge n$ is out of question.

Sagan's criterion~\cite{Sagan1988} applies to the third-order cycle triangle and to the third- and fourth-order
subset triangles, while Dey's extension~\cite{Dey} of Sagan's
criterion yields the fourth-order cycle and fifth-order
subset cases. For the fifth-order cycle triangle and the sixth-order subset triangle, we prove stronger inequalities by induction.

\subsection{The cycle cases for $r=3,4$}
\label{subsec:cycle-r3-r4}

The proof of the log-concavity for the cycle cases $r=3,4$ relies on the recurrence \eqref{rec.cycle} for the higher-order Stirling cycle
numbers, that is, for $n \ge 1$ and $0 \le k \le n$,
\begin{equation}
 \stirlingcycle{n}{k}^{\!(r)}
 =\bigl(n+(r-1)k-1\bigr)\stirlingcycle{n-1}{k}^{(r)}
 +(r-1)!\binom{n+(r-1)k-1}{r-1}
  \stirlingcycle{n-1}{k-1}^{(r)}.
 \label{eq.rec.cyc}
\end{equation}

\noindent 
{\it Proof of Theorem~\ref{thm:main-row-LC} for the cycle case for $r=3,4$.} The log-concavity follows from Lemma~\ref{lem:Sagan-log-concavity} or Lemma~\ref{lem:quadratic-log-concavity}.
For $n\ge3$, set
\[
 a_k=\stirlingcycle{n-1}{k}^{(r)},\qquad
 b_k=\stirlingcycle{n}{k}^{\!(r)}.
\]
Then $b_k=c_ka_k+d_ka_{k-1}$, where for $1\le k\le n$,
\[
 c_k=n+(r-1)k-1,\qquad
 d_k=(r-1)!\binom{n+(r-1)k-1}{r-1}.
\]
We set both coefficients $c_k,d_k$ to zero for $k < 1$ or $k > n$. 

For a given row $n\ge 3$, for $k=1$ and $k=n$, we have
\[
\Delta_k(c)=c_k^2\ge0,\qquad
\Delta_k(d)=d_k^2\ge0,\qquad
\Gamma_k(c,d)=2c_kd_k\ge0,
\]
while all three quantities are equal to zero for $k>n$. Hence the hypotheses of
either criterion in~Lemma~\ref{lem:Sagan-log-concavity} or Lemma~\ref{lem:quadratic-log-concavity} need only be verified for $2\le k\le n-1$.

For $2\le k\le n-1$,
\begin{equation}
 \Delta_k(c)=(r-1)^2>0.
 \label{eq.delta.c.cycle}
\end{equation}
To verify the corresponding inequality for $d$, write
\[
 d_k=\prod_{j=0}^{r-2}d_{k,\,j},\qquad
 d_{k,\,j}=n+(r-1)k-1-j.
\]
For $2\le k\le n-1$ and $0\le j\le r-2$, the three factors
$d_{k-1,\,j},d_{k,\,j},d_{k+1,\,j}$ are positive, and
\[
 d_{k,\,j}^2-d_{k-1,\,j}d_{k+1,\,j}=(r-1)^2.
\]
Thus $d_{k,\,j}^2 \ge d_{k-1,\,j}d_{k+1,\,j}$.
Multiplying these inequalities over $j$ yields that for $2\le k\le n-1$,
\begin{equation}
 \Delta_k(d)\ge0.
 \label{eq.delta.d.cycle}
\end{equation}

Now we complete the proof of the log-concavity for $r=3,4$.
For $r=3$, let $n\ge 3$ and assume that the $(n-1)$st row is log-concave. 
For $2\le k\le n-1$, by a  direct calculation, we see that
\[
 \Gamma_k(c,d)=8(2k+n-2)>0.
\]
Together with~\eqref{eq.delta.c.cycle} and
\eqref{eq.delta.d.cycle}, this confirms the hypotheses of
Lemma~\ref{lem:Sagan-log-concavity}, and hence by induction,
the $n$th row is log-concave. 

For $r=4$, $n\ge3$, and $2\le k\le n-1$, we have
\[
 \Gamma_k(c,d)=-18(9k+3n-14)<0.
\]
Thus Lemma~\ref{lem:Sagan-log-concavity} does not apply to this 
case. However, a direct factorization gives
\begin{equation*}
 \Gamma_k(c,d)^2-4\Delta_k(c)\Delta_k(d)
 =-972(3k+n-4)^2(3k+n-2)(3k+n+2)<0.
\end{equation*}
Together with~\eqref{eq.delta.c.cycle} and
\eqref{eq.delta.d.cycle}, we find that the hypotheses of
Lemma~\ref{lem:quadratic-log-concavity} are satisfied. 
By induction, the $n$th row is
log-concave.
\qed

\subsection{The cycle case for $r=5$}
\label{sub:C_5}

We now consider the cycle case of Theorem~\ref{thm:main-row-LC} for $r=5$.
Notice that
the criterion in
Lemma~\ref{lem:quadratic-log-concavity} does not apply.
For $n=3$ and $k=2$, we get  
\[\Gamma_2(c,d)=-48384<0, \qquad
 \Gamma_2(c,d)^2-4\Delta_2(c)\Delta_2(d)=1268822016>0.\]
Instead, we seek a stronger log-concavity inequality that can be proved by induction. This strategy has
been used in various contexts; see Chen--Wang--Xia~\cite{Chen-Wang-Xia-2011}, Kauers--Paule~\cite{Kauers-Paule-2007}, McNamara--Sagan~\cite{McNamara-Sagan-2010} and Sibuya~\cite{Sibuya-1988}. It turns out that in our case, ChatGPT 5.6 was able to accomplish this task.

 For $n\ge 1$ and
$1 \le k \le n$, write
\begin{equation}
\label{eq:C5-normalization}
C_{n,\,k}=\stirlingcycle{n}{k}^{\!(5)}, \qquad \widetilde{C}_{n,\,k}=\frac{C_{n,\,k}}{(n+4k)!},
\end{equation}
and set $\widetilde C_{n,\,k}=0$ for $k<1$ or $k>n$.
Dividing~\eqref{eq.rec.cyc}  by $(n+4k)!$ yields for $n \ge 2$ and $1 \le k \le n$,
\begin{equation}
\label{eq:A-recurrence}
 \widetilde{C}_{n,\,k}
 =\alpha_{n,\,k}\widetilde{C}_{n-1,\,k}
  +\beta_{n,\,k}\widetilde{C}_{n-1,\,k-1},
\end{equation}
where
\begin{align}
\label{eq.alpha.beta.cycle}
    \alpha_{n,\,k}=\frac{n+4k-1}{n+4k},\qquad \beta_{n,\,k}=\frac1{n+4k}.
\end{align}
Put $N=n+4k$ and define
\begin{equation*}
 \rho^{\mathrm C}_{n,\,k}
 =\frac{(N-4)!(N+4)!}{(N!)^2}.
\end{equation*}
By~\eqref{eq:C5-normalization}, the desired inequality
$C_{n,\,k}^2\ge C_{n,\,k-1}C_{n,\,k+1}$ is equivalent to
\begin{equation}
\label{eq:A-required-strong-LC}
 \widetilde{C}_{n,\,k}^2
 \ge\rho^{\mathrm C}_{n,\,k}
    \widetilde{C}_{n,\,k-1}\widetilde{C}_{n,\,k+1}.
\end{equation}
However, the following stronger version of~\eqref{eq:A-required-strong-LC}
is easier to justify by an inductive argument.

\begin{theorem}
\label{thm.stronger.c}
For $n\ge 3$ and  $2\le k\le n-1$, we have
\begin{align}
\label{eq:A-Lambda-LC}
 \widetilde{C}_{n,\,k}^2
 \ge\lambda^{\mathrm C}_{n,\,k}
    \widetilde{C}_{n,\,k-1}\widetilde{C}_{n,\,k+1},
\end{align}
where
\begin{equation}
\label{eq:C5-Lambda}
 \lambda^{\mathrm C}_{n,\,k}
 =\frac{(n-k+1)(2kn+23k^2-3k-26)}
        {(k-1)(n-k)(2n+23k-26)}.
\end{equation}
\end{theorem}

To see that the above inequality implies 
the log-concavity, we show that $\lambda^{\mathrm C}_{n,\,k}$ is greater than
the factor $\rho^{\mathrm C}_{n,\,k}$ in~\eqref{eq:A-required-strong-LC}.

\begin{lemma}
\label{lem:Lambda-dominates-R}
For $n\ge3$ and $2\le k\le n-1$, we have
\[
 \lambda^{\mathrm C}_{n,\,k}\ge \rho^{\mathrm C}_{n,\,k}.
\]
\end{lemma}

\proof
For $n\ge3$ and $2\le k\le n-1$, put $q=n-k\ge1$ and $N=n+4k$. The difference
$\lambda^{\mathrm C}_{n,\,k}-\rho^{\mathrm C}_{n,\,k}$ has a positive common denominator
\[
 (k-1)qN(N-1)(N-2)(N-3)(2n+23k-26).
\]
After the substitution $n=k+q$, the corresponding numerator becomes
\begin{align*}
 P^{\mathrm C}_{k,\,q}
 =&15625k^6-6250k^5q-20625k^5\\
 &\qquad -4000k^4q^2+4250k^4q-7125k^4\\
&\qquad +600k^3q^3+5300k^3q^2+15450k^3q+17925k^3\\
&\qquad +445k^2q^4+150k^2q^3+6375k^2q^2
   -6500k^2q-7060k^2\\
 & \qquad +58kq^5-275kq^4+300kq^3-6645kq^2
   -7338kq+780k\\
 &\qquad +2q^6-32q^5-60q^4-940q^3-1382q^2-468q.
\end{align*}
To show that $P^{\mathrm C}_{k,\,q}$ is nonnegative 
for $k\ge2$ and $q\ge1$, we may make certain
substitutions to transform the polynomial $P^{\mathrm C}_{k,\,q}$ 
into
a polynomial with nonnegative coefficients in nonnegative parameters. 
The specific substitutions for various regions are given
in the following table. For example, for $2 \le q \le k$, via the substitutions $k=q+a$ and $q=b+2$,
$P^{\mathrm C}_{k,\,q}$ becomes a polynomial in $a$ and $b$ with nonnegative coefficients, where $a \ge 0$ and $b \ge 0$.

\begin{center}
\begin{tabular}{|l|l|}
\hline
Range & Substitution \\ \hline
$q=1$ & $k=a+2$ \\
$2\le q\le k$ & $k=q+a,\quad q=b+2$ \\
$k<q<15k/8$ &
$q=k+t,\quad k=8t/7+a,\quad t=b+1$ \\
$q\ge15k/8$ & $q=15k/8+b,\quad k=a+2$ \\
\hline
\end{tabular}
\end{center}

Thus the numerator of $\lambda^{\mathrm C}_{n,\,k}-\rho^{\mathrm C}_{n,\,k}$ is nonnegative for $n \ge 3$ and $2 \le k \le n -1$, and so the proof is complete.
\qed

In order to prove \eqref{eq:A-Lambda-LC} by induction, 
as the first step, we obtain an upper bound of the ratio $\widetilde{C}_{n,\,k-1}/\widetilde{C}_{n,\,k}$.

\begin{lemma}
\label{lem:A-ratio-bound}
For $n\ge2$ and $2\le k\le n$, we have
\[
 \frac{\widetilde{C}_{n,\,k-1}}{\widetilde{C}_{n,\,k}}\le U_{n,\,k},
\]
where
\begin{equation}
\label{eq:C5-ratio-upper-bound}
 U_{n,\,k}
 =\frac{k}{n-k+1}
 \left[
 \frac{25(k-1)}6+\frac{n-k}{2}
 -\frac{4(n-k)}{3(n-k+6)}
 \right].
\end{equation}
\end{lemma}

\proof
Put $\Phi_{n,\,k}=k!\widetilde{C}_{n,\,k}$. By the exponential generating function for the $5$-associated Stirling
cycle numbers, we get
\begin{equation*}
\widetilde{C}_{n,\,k}=\frac1{k!}[z^n]\left(\sum_{j\ge1}\frac{z^j}{j+4}\right)^k.
\end{equation*}
Thus,
\begin{equation*}
 \Phi_{n,\,k}
=\sum_{\substack{e_1+\cdots+e_k=n\\e_j\ge1}}
   w(e_1,\ldots,e_k),
\end{equation*}
where 
\begin{align*}   
w(e_1,\ldots,e_k)=\prod_{j=1}^k\frac1{e_j+4}.
\end{align*}
Since
\begin{align}
\label{eq.c.phi}
 \frac{\widetilde{C}_{n,\,k-1}}{\widetilde{C}_{n,\,k}}
 =k\frac{\Phi_{n,\,k-1}}{\Phi_{n,\,k}},
\end{align}
in order to bound  $\widetilde{C}_{n,\,k-1}/\widetilde{C}_{n,\,k}$,
it suffices to find an upper bound of $\Phi_{n,\,k-1}/\Phi_{n,\,k}$.
To this end, we shall employ a double counting argument 
to establish the following relation
\begin{equation}
\label{eq:merge-split-identity}
\sum_{\substack{e_1+\cdots+e_k=n\\e_i\ge1}}
 w(e_1,\ldots,e_k)
 \sum_{i=1}^{k-1}\mu(e_i,e_{i+1})
 =(n-k+1)\Phi_{n,\,k-1}.
\end{equation}

Given a $k$-part composition $(e_1,\ldots,e_i, e_{i+1},\ldots,e_k)$ of $n$ with   $e_j \ge 1$ for each $j$, one can merge 
two adjacent parts to form a composition $(e_1,\ldots,e_i+e_{i+1},\ldots,e_k)$ of $n$ with $k-1$ parts, where $1\le i\le k-1$.
For positive integers $s,t$, define
\[
 \mu(s,t)=\frac{(s+4)(t+4)}{s+t+4}.
\]
Then we have 
\begin{align*}
w(e_1,\ldots,e_i+e_{i+1},\ldots,e_k)=w(e_1,\ldots,e_i,e_{i+1},\ldots,e_k)\mu(e_i,e_{i+1}).
\end{align*}
On the other hand, fix a $(k-1)$-part composition $(f_1,\ldots,f_{k-1})$ of $n$. For each part, there are $f_j-1$ ways to 
split $f_j$ into two positive parts. In particular, if $f_j=1$, then there is 
no way to split it into two positive parts. Clearly, the number of
ways to retrieve a $k$-part composition of $n$ by splitting one part of a $(k-1)$-part composition $(f_1, f_2, \ldots, f_{k-1})$ of $n$ equals 
\[
 \sum_{j=1}^{k-1}(f_j-1)=n-k+1. \]
Therefore, \eqref{eq:merge-split-identity} is verified.

We move on to deduce an upper bound of 
\[ \sum_{i=1}^{k-1}\mu(e_i,e_{i+1}).\]
Write $e_i=1+x_i$, where $x_i\ge0$, and put
\[
 q=\sum_{i=1}^k x_i=n-k,
 \qquad t=x_1+x_k.
\]
A direct calculation shows that
\[
 \mu(1+x,1+y)
 =\frac{25}{6}+\frac{x+y}{4}
 -\frac{3(x-y)^2+8(x+y)}{12(x+y+6)}.
\]
Since $\sum_{i=1}^{k-1}(x_i+x_{i+1})=2q-t$ and
$x_i+x_{i+1}\le q$, we find that
\begin{align*}
 &
 \sum_{i=1}^{k-1}
 \frac{3(x_i-x_{i+1})^2+8(x_i+x_{i+1})}
      {12(x_i+x_{i+1}+6)}\\
 &\qquad\ge\frac{2(2q-t)}{3(q+6)}\\
 &\qquad\ge\frac{4q}{3(q+6)}-\frac{t}{4}.
\end{align*}
The last inequality follows from the assumption 
$t\ge 0$ and the fact that
\[
 \frac14-\frac{2}{3(q+6)}=\frac{3q+10}{12(q+6)}>0.
\]
Consequently,
\[
 \sum_{i=1}^{k-1}\mu(e_i,e_{i+1})
 \le\frac{25(k-1)}6+\frac q2-\frac{4q}{3(q+6)}.
\]
Substituting this bound into~\eqref{eq:merge-split-identity} yields 
\[
 \left(\frac{25(k-1)}6+\frac q2-\frac{4q}{3(q+6)}\right)\Phi_{n,\,k}\ge(n-k+1)\Phi_{n,\,k-1}.\]
By \eqref{eq.c.phi}, we get
an upper bound of $\widetilde{C}_{n,\,k-1}/\widetilde{C}_{n,\,k}$, as required. 
\qed

As the second step in the inductive proof of~\eqref{eq:A-Lambda-LC}, we 
prepare the following inequalities.

\begin{lemma}
\label{lem:Lambda-quadratic}
Given $n \ge 4$, let $\alpha_{n,\,k}$ and $\beta_{n,\,k}$ be the coefficients defined in~\eqref{eq.alpha.beta.cycle}, and let $\lambda^{\mathrm C}_{n,\,k}$ and $U_{n,\,k}$ be the functions defined in~\eqref{eq:C5-Lambda} and~\eqref{eq:C5-ratio-upper-bound}, respectively. 
For $3 \le k \le n-2$, define
\begin{align}
 Q_{n,\,k}(u)
 =&(\alpha_{n,\,k}+\beta_{n,\,k}u)^2\notag\\
 &-\lambda^{\mathrm C}_{n,\,k}
 \left(\alpha_{n,\,k-1}
       +\frac{\beta_{n,\,k-1}}{\lambda^{\mathrm C}_{n-1,\,k-1}}u\right)
 \left(\beta_{n,\,k+1}u
       +\frac{\alpha_{n,\,k+1}}{\lambda^{\mathrm C}_{n-1,\,k}}\right).
 \label{eq:Qnk-interior}
\end{align}
For the cases $k=2$ and $k=n-1$, define, respectively,
\begin{align}
 Q_{n,\,2}(u)
 =&(\alpha_{n,\,2}+\beta_{n,\,2}u)^2\notag\\
 &-\lambda^{\mathrm C}_{n,\,2}\alpha_{n,\,1}
 \left(\beta_{n,\,3}u+\frac{\alpha_{n,\,3}}{\lambda^{\mathrm C}_{n-1,\,2}}\right),
 \label{eq:Qnk-lower-boundary}\\
 Q_{n,\,n-1}(u)
 =&(\alpha_{n,\,n-1}+\beta_{n,\,n-1}u)^2\notag\\
 &-\lambda^{\mathrm C}_{n,\,n-1}\beta_{n,\,n}u
 \left(\alpha_{n,\,n-2}
       +\frac{\beta_{n,\,n-2}}{\lambda^{\mathrm C}_{n-1,\,n-2}}u\right).
 \label{eq:Qnk-upper-boundary}
\end{align}
Then, for $n \ge 4$ and $2\le k\le n-1$, if $0\le u\le U_{n-1,\,k}$, then
\[
 Q_{n,\,k}(u)\ge0.
\]
\end{lemma}

\proof
To prove the lemma, we proceed to verify
\begin{equation}
\label{eq:C5-three-quadratic-tests}
 [u^2]Q_{n,\,k}(u)>0,
 \qquad
 Q_{n,\,k}'(U_{n-1,\,k})<0,
 \qquad
 Q_{n,\,k}(U_{n-1,\,k})>0.
\end{equation}
 The first inequality ensures the strict convexity of $Q_{n,\,k}(u)$, so $Q_{n,\,k}'(u)$ is increasing. Thus,  for $0\le u\le U_{n-1,\,k}$, 
\[Q_{n,\,k}'(u) \le Q_{n,\,k}'(U_{n-1,\,k})<0. \]
It follows that
\[Q_{n,\,k}(u)\ge Q_{n,\,k}(U_{n-1,\,k}) >0.\]

First,  consider the range $3\le k\le n-2$. Put $q=n-k\ge2$ and $N=n+4k$.
Then  $Q_{n,\,k}(u)$ has a positive common denominator
\begin{align*}
 D^{\mathrm C}_{n,\,k}
 =&N^2(N-4)(N+4)(k-1)q^2(2n+23k-26)\notag\\
 &\qquad \times (23k^2+2kn-51k-2n+2)
        (23k^2+2kn-5k-26),
\end{align*}
which is positive in this range. Define
\begin{align*}
 \mathcal P^{\mathrm C}_0(n,k)
   &=[u^2]D^{\mathrm C}_{n,\,k}Q_{n,\,k}(u),\\[6pt]
 \mathcal P^{\mathrm C}_1(n,k)
   &=-3(n-k+5)D^{\mathrm C}_{n,\,k}
      Q_{n,\,k}'(U_{n-1,\,k}),\\[6pt]
 \mathcal P^{\mathrm C}_2(n,k)
   &=9(n-k+5)^2D^{\mathrm C}_{n,\,k}
      Q_{n,\,k}(U_{n-1,\,k}),
\end{align*}
all of which can be simplified to polynomials in $n,k$.

Let $k=a+3$ and $q=b+2,$ where $a,b\ge0$. After the substitutions, the polynomials $\mathcal P^{\mathrm C}_0$ and $\mathcal P^{\mathrm C}_1$ have nonnegative coefficients and positive constant terms. 
Then we obtain the first two inequalities in~\eqref{eq:C5-three-quadratic-tests} .

To deal with $\mathcal P^{\mathrm C}_2$, we 
make use of the following substitutions depending
on subdivisions of the range, where $a,b\ge0$.
\begin{center}
    \begin{tabular}{|l|l|}
        \hline
    Range&Substitution
    \\ \hline
$q\le k/2$&$q=b+2,\quad k=a+2b+4$
\\
$q\ge k/2$&$k=a+3,\quad q=(a+3)/2+b$
\\\hline
    \end{tabular}
\end{center}

In each case, we get a polynomial with nonnegative coefficients and 
a positive constant term.
So we reach the conclusion that
$\mathcal P^{\mathrm C}_2(n,k)>0$, which implies 
the third inequality in~\eqref{eq:C5-three-quadratic-tests}.
This completes the proof of Lemma~\ref{lem:Lambda-quadratic} for~\eqref{eq:Qnk-interior}.

For $k=2$, we have
\begin{align*}
 &[u^2]Q_{n,\,2}(u)=\frac1{(n+8)^2}>0,\\[6pt]
 &Q_{n,\,2}'(U_{n-1,\,2})
 =-\frac{2n(3n^4+40n^3+170n^2+583n+2064)}
 {3(n-2)(n+3)(n+4)(n+8)^2(n+10)(n+12)}<0,
\end{align*}
and
\[
 Q_{n,\,2}(U_{n-1,\,2})
 =\frac{p_2(n)}
 {9(n-2)^2(n+3)^2(n+4)(n+8)^2(n+10)(n+12)(n+14)},
\]
where
\begin{align*}
 p_2(n)
=&27n^8+1212n^7+22330n^6+196630n^5+767474n^4\\
 &  \qquad+719996n^3-1835445n^2-2266704n+2566080.
\end{align*}
Writing $x=n-4\ge0$ gives
\begin{align*}
p_2(n)=&27x^8+2076x^7+68362x^6+1236550x^5\\
 &\qquad +13257994x^4+85450588x^3+321231835x^2\\
 &\qquad +641254216x+521124912>0.
\end{align*}
All the factors in $ Q_{n,\,2}(U_{n-1,\,2})$ are positive for $n\ge4$, so the three required inequalities hold for $k=2$. So the case for $Q_{n,\,2}$ in~\eqref{eq:Qnk-lower-boundary} 
is proved.

Finally, for $k=n-1$, we have
\begin{align*}
 [u^2]Q_{n,\,n-1}(u)
 &=\frac{4(5000n^3-30125n^2+54954n-29192)}
 {5(n-2)(5n-4)^2(5n-8)(25n-49)(n-1)(25n-76)},\\[6pt]
 Q_{n,\,n-1}'(U_{n-1,\,n-1})
 &=-\frac{4(1250n^3-3800n^2-5085n+15604)}
 {15(n-2)(5n-8)(5n-4)(25n-76)(25n-49)},\\[6pt]
 Q_{n,\,n-1}(U_{n-1,\,n-1})
 &=\frac{5(n-1)(125n^2-505n+367)}
 {9(5n-8)(25n-76)(25n-49)}.
\end{align*}
After we write $x=n-4\ge0$ and factor out the displayed minus sign in the second expression, all the numerators and denominators in the three expressions become polynomials in $x\ge 0$ with nonnegative coefficients and positive constant terms. Thus we have finished the
case for~\eqref{eq:Qnk-upper-boundary}, and so
the proof is complete. 
\qed

We are now ready to prove \eqref{eq:A-Lambda-LC}.

\noindent{\it Proof of Theorem~\ref{thm.stronger.c}.}
We proceed by induction on $n$. It is readily checked 
that the theorem holds for $n=3$. Assume that
it holds for $n-1$ with $n\geq 4$, that is, for $2\le k\le n-2$,
\begin{align}
\label{eq:A-Lambda-n-1}
 \widetilde{C}_{n-1,\,k}^2
 \ge\lambda^{\mathrm C}_{n-1,\,k}
    \widetilde{C}_{n-1,\,k-1}\widetilde{C}_{n-1,\,k+1}.
\end{align}

Consider the $n$th row of the
fifth-order Stirling cycle triangle. 
For notational convenience,  
write
\[
 a_k=\widetilde{C}_{n-1,\,k},
 \qquad b_k=\widetilde{C}_{n,\,k},
\]
and put
\[
 u=\frac{a_{k-1}}{a_k}.
\]
Both $a_{k-1}$ and $a_{k}$ in this ratio are positive for $2\le k\le n-1$.
Moreover, by Lemma~\ref{lem:A-ratio-bound}, $u$ is bounded by
\[
 0<u\le U_{n-1,\,k}.
\]

First, consider the range $3\le k\le n-2$.
From the induction hypothesis~\eqref{eq:A-Lambda-n-1}, we have
\[
 \frac{a_{k-2}}{a_{k-1}}
 \le\frac{u}{\lambda^{\mathrm C}_{n-1,\,k-1}},
 \qquad
 \frac{a_{k+1}}{a_k}
 \le\frac1{\lambda^{\mathrm C}_{n-1,\,k}u}.
\]
Using the recurrence~\eqref{eq:A-recurrence},
we  obtain
\begin{align*}
 b_{k-1}
 &\le a_ku\left(\alpha_{n,\,k-1}
       +\frac{\beta_{n,\,k-1}}{\lambda^{\mathrm C}_{n-1,\,k-1}}u\right),\\
 b_k
 &=a_k(\alpha_{n,\,k}+\beta_{n,\,k}u),\\
 b_{k+1}
 &\le a_k\left(\beta_{n,\,k+1}
       +\frac{\alpha_{n,\,k+1}}{\lambda^{\mathrm C}_{n-1,\,k}u}\right).
\end{align*}
Therefore, by Lemma~\ref{lem:Lambda-quadratic},
\[
 b_k^2-\lambda^{\mathrm C}_{n,\,k}b_{k-1}b_{k+1}
 \ge a_k^2Q_{n,\,k}(u)\ge0.
\]

For $k=2$, we have $a_0=0$, and hence
\[
 b_1=\alpha_{n,\,1}a_1=\alpha_{n,\,1}a_2u.
\]
The induction hypothesis~\eqref{eq:A-Lambda-n-1} for $k=2$ implies the bound of $b_3$ as given above, i.e.,
\[
 b_3\le a_2\left(\beta_{n,\,3}
       +\frac{\alpha_{n,\,3}}{\lambda^{\mathrm C}_{n-1,\,2}u}\right).
\]
Together with the recurrence  $b_2=a_2(\alpha_{n,\,2}+\beta_{n,\,2}u)$, we find that
\[
 b_2^2-\lambda^{\mathrm C}_{n,\,2}b_1b_3
 \ge a_2^2Q_{n,\,2}(u)\ge0.
\]

For $k=n-1$, we have $a_n=0$, and so
\[
 b_n=\beta_{n,\,n}a_{n-1}.
\]
The induction hypothesis~\eqref{eq:A-Lambda-n-1} for $k=n-2$ implies the bound of $b_{n-2}$ as given above, i.e.,
\[
 b_{n-2}\le a_{n-1}u
 \left(\alpha_{n,\,n-2}
       +\frac{\beta_{n,\,n-2}}{\lambda^{\mathrm C}_{n-1,\,n-2}}u\right).
\]
Using $b_{n-1}=a_{n-1}(\alpha_{n,\,n-1}+\beta_{n,\,n-1}u)$, we obtain
\[
 b_{n-1}^2-\lambda^{\mathrm C}_{n,\,n-1}b_{n-2}b_n
 \ge a_{n-1}^2Q_{n,\,n-1}(u)\ge0.
\]
Thus~\eqref{eq:A-Lambda-LC} holds for any $n\ge3$ and
$2\le k\le n-1$.
\qed

In summary, we may claim this completes the proof of Theorem~\ref{thm:main-row-LC} for the cycle case $r=5$.

\subsection{The subset cases for $3\leq r \leq 5$}
\label{subsec:subset-log-concavity}

We now turn to the proof of the row log-concavity for
the $r$th-order Stirling subset
triangles  for $3\leq r \leq 5$. The proof relies
on the recurrence relation~\eqref{rec.subset} for higher-order Stirling subset numbers, namely, for $n \ge1$ and $0 \le k \le n$,
\begin{align}
\stirlingsubset{n}{k}^{\!(r)}
=
k\stirlingsubset{n-1}{k}^{(r)}
+
\binom{n+(r-1)k-1}{r-1}
\stirlingsubset{n-1}{k-1}^{(r)}.
\label{eq:subset-rec-log}
\end{align}

\noindent 
{\it Proof of Theorem~\ref{thm:main-row-LC} for the subset case for $3\leq r \leq 5$.}  For $3 \leq r \leq 5$ and $n\ge 3$, set
\begin{align*}
a_k=\stirlingsubset{n-1}{k}^{(r)},
\qquad
b_k=\stirlingsubset{n}{k}^{\!(r)}.
\end{align*}
Then
\begin{align*}
b_k=c_ka_k+d_ka_{k-1},
\end{align*}
where, for $1\le k\le n$,
\begin{align*}
c_k=k,
\qquad
d_k=\binom{n+(r-1)k-1}{r-1}.
\end{align*}
We set $c_k$ and $d_k$ to zero outside this range.

It suffices to verify the hypotheses of either Lemma~\ref{lem:Sagan-log-concavity} or  Lemma~\ref{lem:quadratic-log-concavity}   for $2\le k\le n-1$. Assuming that
$k$ is in this range, we have
\begin{equation}
 \Delta_k(c)=1.
 \label{eq:delta.c.subset}
\end{equation}
Moreover, the sequence $d$ differs from its cycle counterpart only
by the positive constant factor $1/(r-1)!$. Thus inequality~\eqref{eq.delta.d.cycle}  for the cycle case 
implies that for $2\le k\le n-1$,
\begin{equation}
 \Delta_k(d)\ge0.
 \label{eq:delta.d.subset}
\end{equation}

For $r=3,4$, we proceed by induction on $n$. Let $n\ge3$, and assume 
that the $(n-1)$st row is log-concave.
For $r=3$ and $2\le k\le n-1$, we have
\begin{align*}
\Gamma_k(c,d)=2(2k+2n-3)>0.
\end{align*}
For $r=4$ and $2\le k\le n-1$, we find that
\begin{align*}
\Gamma_k(c,d)
=
9k(n-2)+3(n-2)^2+8>0.
\end{align*}
Moreover, we know that
$\Delta_k(c) \ge 0$ and $\Delta_k(d)\ge 0$, as shown in~\eqref{eq:delta.c.subset} and~\eqref{eq:delta.d.subset}.
Hence Lemma~\ref{lem:Sagan-log-concavity} says that the $n$th rows of both the third- and the fourth-order Stirling subset triangles are log-concave.

For $r=5$, when $n=3$ and $k=2$, we have
\begin{align*}
\Gamma_2(c,d)=-206<0.
\end{align*}
Thus Sagan's criterion no longer applies. 
Instead, we use Dey's criterion in Lemma~\ref{lem:quadratic-log-concavity} to 
deduce the row log-concavity.
We proceed by induction on $n$ as before. Let $n \ge 3$ and assume that the $(n-1)$st row is log-concave. To apply
Lemma~\ref{lem:quadratic-log-concavity}, by~\eqref{eq:delta.c.subset} and
\eqref{eq:delta.d.subset}, it remains to verify that for $2\le k\le n-1$,
\eqref{eq:quadratic-discriminant} holds when $\Gamma_k(c,d)<0$. Suppose that $\Gamma_k(c,d)<0$. A direct calculation shows that
\[
4\Delta_k(c)\Delta_k(d)-\Gamma_k(c,d)^2
=P_{n,\,k}\left(Q_{n,\,k}-\frac{n-1}{2}\Gamma_k(c,d)\right),
\]
where
\begin{align*}
 P_{n,\,k}=2\left(n+4k-\frac52\right)^2+\frac92>0,
\end{align*}
and
\begin{align*}
Q_{n,\,k}=&64k^3n-192k^3+32k^2n^2-160k^2n+272k^2\\
&\qquad+4kn^3-24kn^2-10kn
+158k+n^3-28n^2+136n-253.
\end{align*}
Set $a=k-2\ge0$ and $b=n-3\ge0$. Then $Q_{n,\,k}$ becomes
\begin{align*}
&64a^3b+32a^2b^2+416a^2b+80a^2+4ab^3+140ab^2 +850ab+340a \\
&\qquad+ 9b^3+133 b^2+543b+290>0.
\end{align*}
Since $n\ge3$ and $\Gamma_k(c,d)<0$, it follows that
\[
\Gamma_k(c,d)^2<4\Delta_k(c)\Delta_k(d).
\]
Using Lemma~\ref{lem:quadratic-log-concavity}, we complete the proof by 
induction.
\qed

\subsection{The subset case for $r=6$}
\label{sub:S_6}

Unlike the row log-concavity behavior of the \(r\)th-order Stirling subset triangles for \(1\le r\le5\), the case \(r=6\) is more intricate. First, the
row log-concavity holds for $n\ge 23$. Second, we
need a stronger version of the log-concavity inequality
that admits a justification by induction.

Write
$S_{n,\,k}=\stirlingsubset{n}{k}^{\!(6)}$. For $n\ge1$ and
$1\le k\le n$, define
\begin{equation}
\label{eq:S6-normalization}
 \widetilde{S}_{n,\,k}
 =\frac{120^k k!(k-1)!(n-k)!}{(n+5k)!}S_{n,\,k},
\end{equation}
and set $\widetilde S_{n,\,k}=0$ for $k<1$ or $k>n$. By
\eqref{eq:subset-rec-log}, we have for $n\ge2$ and
$1 \le k \le n$,
\begin{equation}
\label{eq:S6-normalized-recurrence}
\widetilde{S}_{n,\,k}=\alpha_{n,\,k}\widetilde{S}_{n-1,\,k}+\beta_{n,\,k}\widetilde{S}_{n-1,\,k-1},
\end{equation}
where
\begin{equation}
\label{eq:S6-recurrence-coefficients}
 \alpha_{n,\,k}=\frac{k(n-k)}{n+5k},
 \qquad
 \beta_{n,\,k}=\frac{k(k-1)}{n+5k}.
\end{equation}

Put $q=n-k$, $N=n+5k$, and define
\begin{equation*}
 \rho^{\mathrm S}_{n,\,k}
 =\frac{(k-1)q(N-5)!(N+5)!}{(k+1)(q+1)(N!)^2}.
\end{equation*}
By~\eqref{eq:S6-normalization}, the log-concavity   for $S_{n,\,k}$ is
equivalent to
\begin{equation}
\label{eq:S6-required-inequality}
 \widetilde{S}_{n,\,k}^2 \ge \rho^{\mathrm S}_{n,\,k}\widetilde{S}_{n,\,k-1}\widetilde{S}_{n,\,k+1}.
\end{equation}

We shall establish an auxiliary inequality that implies~\eqref{eq:S6-required-inequality} for all sufficiently large $n$.

\begin{theorem}
\label{thm.stronger.s}
Define
\begin{equation}
\label{eq:S6-lambda}
 \lambda^{\mathrm S}_{n,\,k}
 =
 1+\frac{52n(n-2)}{(n+5k)^3}.
\end{equation}
Then, for $n\ge3$ and $2\le k\le n-1$, we have
\begin{equation}
\label{eq:A-Lambda-LS}
 \widetilde{S}_{n,\,k}^2
 \ge
 \lambda^{\mathrm S}_{n,\,k}
 \widetilde{S}_{n,\,k-1}\widetilde{S}_{n,\,k+1}.
\end{equation}
\end{theorem}

By the following relation, we see that the above inequality implies the log-concavity for $n \ge 80$.

\begin{lemma}
\label{lem:S6-factor}
For $n\ge80$ and $2\le k\le n-1$,
\[
 \lambda^{\mathrm S}_{n,\,k}>\rho^{\mathrm S}_{n,\,k}.
\]
\end{lemma}

\proof
Put $N=n+5k$. For every integer $N>12$, we have
\begin{equation*}
 \frac{(N-5)!(N+5)!}{(N!)^2}<\frac{N+13}{N-12},
\end{equation*}
because
\[
 \frac{N+13}{N-12}
 -\frac{(N-5)!(N+5)!}{(N!)^2}
 =\frac{120(2N+1)(5N^2+5N+12)}
 {(N-12)\prod_{j=0}^{4}(N-j)}>0.
\]
Consequently, to prove the lemma, it suffices to show that
\[
 \left(1+\frac{52n(n-2)}{N^3}\right)\frac{(k+1)(n-k+1)}{(k-1)(n-k)}
 -\frac{N+13}{N-12}>0.
\]
Multiplying both sides by the positive factor $N^3(k-1)(n-k)(N-12)$, we see that it suffices to prove that 
\begin{align*}
 P^{\mathrm S}_{n,\,k}
 =&[N^3+52n(n-2)](k+1)(n-k+1)(N-12)\\
 &\qquad -N^3(k-1)(n-k)(N+13)>0.
\end{align*}
Write $n=80+x$. We have
\begin{equation*}
 1638400\,P^{\mathrm S}_{n,\,k}
 =(38x-80k+3055)^2J_{x,\,k}+kF_x+G_x,
\end{equation*}
where $J_{x,\,k}$ is a polynomial in $x$ and $k$, while $F_x$ and $G_x$ are polynomials in $x$. All three polynomials have nonnegative coefficients and a positive constant term.
It is now clear that $P^{\mathrm S}_{n,\,k}>0$ for    $n\ge 80$ and $k\ge2$.
This completes the proof.
\qed

The following lemma is
needed in the proof of \eqref{eq:A-Lambda-LS}.

\begin{lemma}
\label{lem:S6-quadratic}
Let $\alpha_{n,\,k}$ and $\beta_{n,\,k}$ be the coefficients defined in~\eqref{eq:S6-recurrence-coefficients}, and let $\lambda^{\mathrm S}_{n,\,k}$ be the function defined in~\eqref{eq:S6-lambda}. For $n\ge3$ and $2\le k\le n-1$, define
\begin{align}
 Q^{\mathrm S}_{n,\,k}(u)
 =&(\alpha_{n,\,k}+\beta_{n,\,k}u)^2\notag\\
 &-\lambda^{\mathrm S}_{n,\,k}
 \left(\alpha_{n,\,k-1}
       +\frac{\beta_{n,\,k-1}}{\lambda^{\mathrm S}_{n-1,\,k-1}}u\right)
 \left(\beta_{n,\,k+1}u
       +\frac{\alpha_{n,\,k+1}}{\lambda^{\mathrm S}_{n-1,\,k}}\right).
 \label{eq:S6-quadratic}
\end{align}
Then for $u\in\mathbb R$,
\[
 Q^{\mathrm S}_{n,\,k}(u)>0.
\]
\end{lemma}

\proof
 Unlike in the proof of Lemma~\ref{lem:Lambda-quadratic},
we do not need to bound $u$.
For $n\ge3$ and $2\le k\le n-1$, put $q=n-k\ge1$ and $N=n+5k=6k+q$. For notational convenience, write
\[
 H_{n,\,k}=(n+5k)^3+52n(n-2),
\]
and
\begin{equation*}
 D^{\mathrm S}_{n,\,k}
 =k^2N(N^2-25)H_{n-1,\,k}H_{n-1,\,k-1}.
\end{equation*}
Both polynomials are positive for   $n\ge3$ and $2\le k\le n-1$. In this notation, we have 
\begin{align*}
    \lambda^{\mathrm S}_{n,\,k}=H_{n,\,k}/N^3, \quad \alpha_{n,\,k}=\frac{kq}{N}, \quad \beta_{n,\,k}=\frac{k(k-1)}{N}.
\end{align*}
Now \eqref{eq:S6-quadratic} takes the form
\begin{equation}
\label{eq:S6-quadratic-coefficients}
 Q^{\mathrm S}_{n,\,k}(u)
 =\frac{k^2}{N^2D^{\mathrm S}_{n,\,k}}
  \bigl(\mathcal L_{n,\,k}u^2
       +\mathcal M_{n,\,k}u+\mathcal C_{n,\,k}\bigr).
\end{equation}
The three coefficients are
\begin{align*}
 \mathcal L_{n,\,k}
 =&(k-1)^2D^{\mathrm S}_{n,\,k}-H_{n,\,k}(k^2-1)k(k-2)(N-6)^3H_{n-1,\,k},\\
 \mathcal M_{n,\,k}
 =&2q(k-1)D^{\mathrm S}_{n,\,k} -H_{n,\,k}(k^2-1)\\
 & \qquad \times 
 \bigl[k(q+1)H_{n-1,\,k}H_{n-1,\,k-1}+(k-2)(q-1)(N-1)^3(N-6)^3\bigr],\\
 \mathcal C_{n,\,k}
 =&q^2D^{\mathrm S}_{n,\,k}-H_{n,\,k}(k^2-1)(q^2-1)(N-1)^3H_{n-1,\,k-1}.
\end{align*}

Consider the discriminant of $\mathcal L_{n,\,k}u^2 +\mathcal M_{n,\,k}u+\mathcal C_{n,\,k}$ in~\eqref{eq:S6-quadratic-coefficients}, that is,
\[ \mathcal M_{n,\,k}^{2}-4\mathcal L_{n,\,k}\mathcal C_{n,\,k}
=-4(k^2-1)(n-1)^2H_{n,\,k}\,R_{k,\,n-k},
\]
where 
\begin{equation*}
 R_{k,\,n-k}
 =
 \frac{4\mathcal L_{n,\,k}\mathcal C_{n,\,k}-\mathcal M_{n,\,k}^{2}}
      {4(k^2-1)(n-1)^2H_{n,\,k}}.
\end{equation*}
Since
\[
4(k^2-1)(n-1)^2H_{n,\,k}>0,
\]
it remains to prove that
$R_{k,\,q}>0$ for $k\ge2$ and $q\ge1$.
With the substitutions $k=a+2$ and $q=b+1$, we see that
\begin{equation*}
 R_{k,\,q}
 =1772706816\,a^{11}(a^2-b^2)^2+\mathcal T(a,b),
\end{equation*}
where $\mathcal T$ is a polynomial in $a$ and $b$ with nonnegative coefficients and a positive constant term.  Hence
\begin{equation}
\label{eq:S6-discriminant-positive}
 \mathcal M_{n,\,k}^{2}-4\mathcal L_{n,\,k}\mathcal C_{n,\,k} <0
\end{equation}
for $n\ge3$ and $ 2\le k\le n-1$. The argument is valid for all real
$k\ge2$ and $q\ge1$. Hence this discriminant is negative
throughout this real domain.

To determine the sign of the leading coefficient of $\mathcal L_{n,\,k}u^2+\mathcal M_{n,\,k}u+\mathcal C_{n,\,k}$, we note that
the inequality~\eqref{eq:S6-discriminant-positive} implies that
$\mathcal L_{k+q,\,k}$ is nonzero for every $k\ge2$ and $q\ge1$. Given $q\ge1$,
the function $\mathcal L_{k+q,\,k}$ is continuous on
$k \in [2,\infty)$, and
\[
 \mathcal L_{q+2,\,2}=D^{\mathrm S}_{q+2,\,2}>0.
\]
It follows that $\mathcal L_{k+q,\,k}>0$, or equivalently,  $\mathcal L_{n,\,k}>0$ for $n\ge3$ and $ 2\le k\le n-1$.
Since $Q^{\mathrm S}_{n,\,k}(u)$ equals a positive factor  multiplied by the quadratic polynomial  which is strictly positive, the lemma follows.
\qed

We now give a proof of \eqref{eq:A-Lambda-LS}.

\noindent{\it Proof of Theorem~\ref{thm.stronger.s}.}
We proceed by induction on $n$. It is readily checked that the theorem holds for $n=3$.
Suppose that $n\ge4$ and that~\eqref{eq:A-Lambda-LS} holds for the
$(n-1)$st row, that is, for $2\le k\le n-2$, we have
\begin{align*}
 \widetilde{S}_{n-1,\,k}^2
 \ge\lambda^{\mathrm S}_{n-1,\,k}
    \widetilde{S}_{n-1,\,k-1}\widetilde{S}_{n-1,\,k+1}.
\end{align*}
For $2\le k\le n-1$, write
\[
 a_j=\widetilde{S}_{n-1,\,j},\qquad b_j=\widetilde{S}_{n,\,j},
 \qquad u=\frac{a_{k-1}}{a_k}>0.
\]

By the induction hypothesis and the assumption that for $n \ge 1$, $\widetilde{S}_{n,\,j}=0$ for $j<1$ or $j>n$, we have
\begin{equation*}
 \frac{a_{k-2}}{a_{k-1}}\le\frac{u}{\lambda^{\mathrm S}_{n-1,\,k-1}},
 \qquad
 \frac{a_{k+1}}{a_k}\le\frac1{\lambda^{\mathrm S}_{n-1,\,k}u}.
\end{equation*}
At $k=2$, the first inequality follows from $a_0=0$; at $k=n-1$, the second follows from $a_n=0$.
All remaining cases follow from the induction hypothesis.
Using the recurrence relation~\eqref{eq:S6-normalized-recurrence}, we then obtain
\begin{align*}
 b_{k-1}
 &\le a_ku\left(\alpha_{n,\,k-1}
       +\frac{\beta_{n,\,k-1}}{\lambda^{\mathrm S}_{n-1,\,k-1}}u\right),\\
 b_k&=a_k(\alpha_{n,\,k}+\beta_{n,\,k}u),\\
 b_{k+1}
 &\le a_k\left(\beta_{n,\,k+1}
       +\frac{\alpha_{n,\,k+1}}{\lambda^{\mathrm S}_{n-1,\,k}u}\right).
\end{align*}
By Lemma~\ref{lem:S6-quadratic}, we deduce that
\[
 b_k^2-\lambda^{\mathrm S}_{n,\,k}b_{k-1}b_{k+1}
 \ge a_k^2Q^{\mathrm S}_{n,\,k}(u)>0.
\]
This proves
\eqref{eq:A-Lambda-LS} for  $n\ge3$ and $2\le k\le n-1$. \qed

To conclude this section, we note that all the 
cases of Theorem~\ref{thm:main-row-LC} have been proved, and the range
$n\geq 23$ in Theorem~\ref{thm:main-row-LC} is 
obtained from the condition $n\geq 80$ in Lemma~\ref{lem:S6-factor} along with
numerical verifications for $23 \leq n\le 79$.

\section{The non-real-rootedness for $r=3$}
\label{sec:S3-nonreal}

In this section, we confirm the non-real-rootedness 
of the row-generating polynomial $s_{3,\,n}(x)$ for the
third-order Stirling subset triangle~\cite[Conjecture~1.4(c) and Section~7.1]{Deb-Sokal}. 
It has been shown in~\cite[Section~7.1]{Deb-Sokal} that $s_{r,\,n}(x)$ is not real-rooted for
$r\geq 4$ and $n\geq 3$ and the row-generating polynomial $c_{r,\,n}(x)$ for the
 $r$th-order Stirling cycle triangle is not real-rooted for  $r\ge3$ and
$n\ge3$.

Our proof uses a necessary condition that follows from
Holtz and Tyaglov~\cite[Theorem~4.29]{Holtz-Tyaglov-2012}. An equivalent condition
appears in Nguyen and
Vishnyakova~\cite[Lemma~1.2]{Nguyen-Vishnyakova-2020}.
We include a proof for completeness.

\begin{lemma}
\label{lem:four-leading-coefficients}
Let $f(x)$ be a monic polynomial of degree $m\geq3$, with
\begin{align*}
f(x)
=
x^m+e_1x^{m-1}+e_2x^{m-2}+e_3x^{m-3}+\cdots+e_m.
\end{align*}
If all the zeros of $f(x)$ are nonpositive real numbers, then
\begin{align}
e_1^2e_2+3e_1e_3-4e_2^2\geq0.
\label{eq:four-leading-condition}
\end{align}
\end{lemma}

\proof 
Write
\begin{align*}
f(x)=\prod_{i=1}^{m}(x+\lambda_i),
\end{align*}
where $\lambda_i \ge 0$,
and put
\begin{align*}
p_j=\sum_{i=1}^{m}\lambda_i^j,
\qquad j=1,2,3.
\end{align*}
Applying   the Girard--Newton formula or Waring's formula~\cite{Stanley}, we obtain
\begin{align*}
p_1=e_1,\qquad
p_2=e_1^2-2e_2,\qquad
p_3=e_1^3-3e_1e_2+3e_3.
\end{align*}
Consequently,
\begin{align*}
e_1^2e_2+3e_1e_3-4e_2^2
&=p_1p_3-p_2^2\\[6pt]
&=\sum_{1\leq i<j\leq m}
  \bigl(\lambda_i\lambda_j^3
       +\lambda_j\lambda_i^3
    -2\lambda_i^2\lambda_j^2\bigr)\\[6pt]
&=\sum_{1\leq i<j\leq m}
  \lambda_i\lambda_j(\lambda_i-\lambda_j)^2
\geq0.
\end{align*}
This proves \eqref{eq:four-leading-condition}.
\qed

\begin{theorem}
\label{theorem:S3-nonreal}
For $n\ge4$, the polynomial $s_{3,\,n}(x)$ is not real-rooted.  
\end{theorem}

\proof 
For $n\ge 4$, write
$a_{n,\,j}=\stirlingsubset{n}{n-j}^{(3)}$ for $0\le j\le3$.
By the exponential formula for the $3$-associated Stirling subset numbers, we have
\begin{align}
a_{n,\,j}
=
\frac{(3n-2j)!}{(n-j)!}
[z^j]\left(
\sum_{\ell\ge0}\frac{z^\ell}{(\ell+3)!}
\right)^{n-j}.
\label{eq:S3-leading-extraction}
\end{align}
Since $\stirlingsubset{n}{0}^{(3)}=0$ and $\stirlingsubset{n}{1}^{(3)}=1$, the polynomial
\begin{align}
f_n(x)=\frac{s_{3,\,n}(x)}{a_{n,\,0}x}
\label{eq:f_n}
\end{align}
is monic of degree $n-1$ with positive coefficients.
For $j=1,2,3$, let $e_j=a_{n,\,j}/a_{n,\,0}$ be the
coefficient of $x^{n-1-j}$ in $f_n(x)$.
Extracting coefficients in \eqref{eq:S3-leading-extraction}
gives
\begin{align*}
e_1
&=\frac{n-1}{2(3n-1)},\\
e_2
&=\frac{(n-2)(5n-7)}
        {40(3n-2)(3n-1)},\\
e_3
&=\frac{(n-3)(n-2)(5n^2-21n+20)}
        {80(3n-5)(3n-4)(3n-2)(3n-1)}.
\end{align*}
Note that these expressions can also be obtained from the recurrence~\eqref{rec.subset}.
It follows that
\begin{align*}
e_1^2e_2+3e_1e_3-4e_2^2
=
-\frac{(n-2)P(n)}
{50(3n-5)(3n-4)(3n-2)^2(3n-1)^3},
\end{align*}
where
\[
P(n)=36n^4-222n^3+416n^2-277n+45.
\]
Writing $u=n-4\ge0$, we get
\[
P(n)
=
36u^4+354u^3+1208u^2+1611u+601>0.
\]
Therefore, for
$n\ge4$, 
\begin{align}
e_1^2e_2+3e_1e_3-4e_2^2<0.
\label{eq:contra}
\end{align}

Suppose, to the contrary,  that $s_{3,\,n}(x)$ is real-rooted.
Then $f_n(x)$ is also real-rooted. Since $f_n(x)$ has positive
coefficients, it cannot have any positive real roots.
In view of Lemma~\ref{lem:four-leading-coefficients}, we infer that 
\begin{align*}
e_1^2e_2+3e_1e_3-4e_2^2\geq 0,
\end{align*}
a contradiction to \eqref{eq:contra}.
Thus $s_{3,\,n}(x)$ is not real-rooted.
\qed

An alternative proof of Theorem~\ref{theorem:S3-nonreal} uses the
Hermite--Sylvester criterion~\cite[Theorem~1]{Nathanson}.
For $n\ge4$, let $\lambda_1,\ldots,\lambda_{n-1}$ be the zeros of $f_n(x)$ defined in~\eqref{eq:f_n} and set
\[ p_j=\sum_{\ell=1}^{n-1}\lambda_\ell^j.\]
Consider the principal minor
\begin{align*}
\det\begin{pmatrix}
p_2 & p_3\\
p_3 & p_4
\end{pmatrix},
\end{align*}
which, after the substitution $u=n-4$,  can be expressed as
\begin{align*}
-\frac{2R(u)}
{875(3u+5)(3u+7)^2(3u+8)^2(3u+10)^3(3u+11)^4},
\end{align*}
 where $R(u)$ is a polynomial with   positive
coefficients, including a positive constant term. 
The details of the computations are suppressed. So this principal
minor is negative for every $n\ge4$, implying that 
  $s_{3,\,n}(x)$ is not real-rooted for $n\geq 4$.

\section{An open question}
\label{subsec:row-log-concavity-conjectures}

Computations suggest that there may be a pattern in the log-concavity and log-convexity of the rows of the higher-order Stirling cycle and subset triangles.  We conclude with the following conjecture.

\begin{conjecture}
\label{conj:single-negative-interval}
Let $T=C^{(r)}$ with $r\ge6$, or $T=S^{(r)}$ with $r\ge7$.
For every $n\ge3$, there exist integers
$2\le\ell_n\le u_n\le n-1$ such that the three sequences
\[
 (T_{n,\,k})_{k=1}^{\ell_n},
 \qquad
 (T_{n,\,k})_{k=\ell_n-1}^{u_n+1},
 \qquad
 (T_{n,\,k})_{k=u_n}^{n}
\]
are strictly log-concave, strictly log-convex,
and strictly log-concave, respectively.
\end{conjecture}

\noindent{\textbf{ Declaration of the Use of AI.}} 
The rational functions $\lambda_{n,\,k}^{\mathrm C}$
in~\eqref{eq:C5-Lambda} and
$\lambda_{n,\,k}^{\mathrm S}$
in~\eqref{eq:S6-lambda} were found by ChatGPT 5.6 given the requirements that $\lambda_{n,\,k}\ge \rho_{n,\,k}$ and that the resulting stronger inequalities can be justified by induction.

\end{document}